\documentclass{amsart}
\usepackage{graphicx}
\usepackage{amscd}

\newtheorem{theorem}{Theorem}
\theoremstyle{plain}

\newtheorem{corollary}{Corollary}

\newtheorem{definition}{Definition}
\newtheorem{example}{Example}

\newtheorem{lemma}{Lemma}

\newtheorem{remark}{Remark}

\numberwithin{equation}{section}

\begin{document}
\title[Embedding]{Embedding planar compacta in planar continua with application: homotopic
maps between planar Peano continua are characterized by the fundamental
group homomorphism. }
\author{Paul Fabel}
\address{Drawer MA Mississippi State, MS 39762}
\date{March 19 2009}

\begin{abstract}
The CAT(0) geometry of a planar PL disk (determined by internal paths of
minimal length) is employed to prove every planar compactum with connected
complement can be an embedded in a cellular planar continuum by attaching a
null sequence of arcs with disjoint interiors.

This leads to a proof that two based maps from a planar Peano continuum to a
planar set are homotopic iff they induce the same homomorphism between
fundamental groups.
\end{abstract}

\maketitle

\section{Introduction}

The capacity to recognize homotopic maps plays a central role in classifying
planar continua up to homotopy equivalence.

The second of two main results (Theorems \ref{main1} and \ref{main4})
establishes if $X\subset R^{2}$ is a locally path connected (i.e. Peano)
continuum, and $Y\subset R^{2}$ is arbitrary, and if $f,g:(X,p)\rightarrow
(Y,q)$ are based maps, then $f$ is homotopic to $g$ if and only if the
induced maps $f_{\ast },g_{\ast }:\pi _{1}(X,p)\rightarrow \pi _{1}(Y,q)$
satisfy $f_{\ast }=g_{\ast }.$

The above statement is generally false for planar continua, due for example
to the existence of cellular noncontractible planar continua. Such examples
also confirm the potential failure of the conclusion of Whitehead's Theorem 
\cite{Wh1} \cite{Wh2} ( maps between $CW$ complexes which induce
isomorphisms on homotopy groups are homotopy equivalences)

For planar Peano continua $X$ and $Y,$ it is an open question whether $%
f:X\rightarrow Y$ is a homotopy equivalence precisely if $f_{\ast }$ induces
an isomorphism between fundamental groups.

In the simplest nontrivial case, (the Hawaiian earring $HE$, the union of a
null sequence of circles joined at a common point), the group $\pi _{1}(HE)$
is uncountable and not free \cite{Smit}, it naturally injects into the
inverse limit of finite free groups, and its elements can be understood as
``infinite words'' in generators $x_{1},x_{2},...$ such that each letter
appears finitely many times \cite{MM}. Remarkably, Eda \cite{Eda1} proved
all homomorphisms of the Hawaiian earring group are (up to a change of base
point isomorphism) induced by maps, and hence the self homotopy equivalences 
$f:HE\rightarrow HE$ are precisely the maps such that $f$ induces an
isomorphism of $\pi _{1}(HE).$

More generally, positive answers are emerging in case $X$ is 1 dimensional 
\cite{Eda2} or homotopy equivalent to a 1 dimensional planar Peano continuum 
\cite{CC2}.

Most generally (e.g. \cite{Z2}) , Theorem \ref{main4} of the paper at hand
immediately yields a partial answer\textbf{:} If $X,Y\subset R^{2}$ are
Peano continua, then a map $f:X\rightarrow Y$ is a homotopy equivalence if $%
f_{\ast }$ is an isomorphism and $f_{\ast }^{-1}$ is induced by a map
(Corollary \ref{White}).

To prove Theorem \ref{main4}, we must confront the fact that if $A\subset
int(D^{2})$ is an \textbf{arbitrary} compactum in the closed unit disk $%
D^{2} $ such that $\dim (A)\leq 1,$ and if $\{U_{n}\}\subset
int(D^{2})\backslash A $ is a null sequence of disjoint round open disks
converging (limit supremum) to $A,$ then $X=D^{2}\backslash \{U_{1}\cup
U_{2}...\}$ is a Peano continuum such that $X\backslash int(X)=Z=A\cup
\partial U_{1}\cup \partial U_{2}...$

Since $A$ is arbitrary, critical to the proof of Theorem \ref{main4}, is our
other main result (Theorem \ref{main1}), which establishes that every planar
compactum $Z\subset R^{2}$ (such that $R^{2}\backslash Z$ is connected) can
be embedded in a nonseparating planar continuum $W\subset R^{2}$ by
attaching a null sequence of topological arcs (with disjoint interiors) to $%
Z.$

Theorem \ref{main1} effectively generalizes some recent work of Blokh,
Misiurewicz, and Oversteegen \cite{BMO}, who employ a similar strategy in
the special case the nontrivial components of $Z$ form a null sequence of
Peano continua. In \cite{BMO}, the resulting continuum $W$ is a Peano
continuum, but there is generally no hope for such a conclusion in Theorem 
\ref{main1}, since some components of $Z$ can fail to be locally connected.
On the other hand a new technical hurdle (not relevant in \cite{BMO}) is
that $Z$ can have uncountably many components of large diameter, (for
example if $A$ is the product of $[0,1]$ with a Cantor set).

Our proof of Theorem \ref{main1} exploits the $CAT(0)$ geometry of PL planar
disks as determined by internal paths of minimal Euclidean length. In
various settings we wish to select a collection of short disjoint arcs with
certain properties. The $CAT(0)$ geometery supplies nontransverse arcs which
can be perturbed to be disjoint while satisfying the other desired
properties.

\section{Definitions and notation}

PL denotes \textbf{piecewise linear. }An \textbf{arc} is a topological space
homeomorphic to $[0,1].$ A \textbf{disk} is any space homeomorphic to the
closed round unit disk.

If $\alpha \subset R^{2}$ is an arc then the \textbf{length }$l(\alpha )$ is
the familiar Euclidean arc length ( for example if $\alpha $ is $PL$ then $%
l(\alpha )$ the sum of the lengths of the finitely many concatenated line
segments whose union is $\alpha $).

\begin{definition}
\label{MNdef}Suppose $A\subset R^{2}$ is the union of finitely many pairwise
disjoint closed PL topological disks. Suppose $B\subset A$ and suppose $%
B\cap A_{i}\neq \emptyset $ for each component $A_{i}\subset A.$ Let $%
N(A,B)=\inf \{\delta >0|$ for each $x\in \partial A$ there exists $y\in B$
and a PL arc $\gamma \subset A$ connecting $x$ to $y$ such that $l(\gamma
)<\delta \}.$ Let $M(A,B)=\inf \{\varepsilon >0|$ for each $x\in A$ there
exists $y\in B$ and a PL arc $\gamma \subset A$ connecting $x$ to $y$ such
that $l(\gamma )<\delta .$
\end{definition}

The notation \textbf{interior} is slightly abused (in the context of arcs
and trees) as follows.

If $X$ is a 2 dimensional planar set then $int(X)$ denotes the largest open
set $U\subset R^{2}$ such that $U\subset X.$

If $X$ $\subset R^{2}$ is a 2 dimensional continuum let $Fr(X)=X\backslash
int(X)$ and call $Fr(X)$ the \textbf{frontier} of $X.$

However in the special case $\alpha \subset R^{2}$ is an arc we let $%
\partial \alpha $ denotes the endpoints of $\alpha $ and $int(\alpha
)=\alpha \backslash \partial \alpha .$

If $E\subset R^{2}$ is a closed topological disk then $\partial E$ \ denotes
the simple closed curve bounding $int(E).$ Thus $\partial E=\partial
(int(E)).$

If $X\subset R^{2}$ then $X$ is a\textbf{\ tree} if $X$ is connected and
simply connected and homeomorphic to the union of finitely many straight
Euclidean line segments.

If $T$ is a tree then $int(T)=\{x\in T|T\backslash x$ is not connected$\}$
and $\partial T=T\backslash int(T).$

The tree $T$ is a \textbf{triod} if $T$ is homeomorphic to the planar set $%
([-1,1]\times \{0\})\cup (\{0\}\times \lbrack 0,1])$

If $X$ is a tree then $x$ is an \textbf{endpoint} of $X$ if $X\backslash
\{x\}$ is connected.

If $X\subset R^{2}$ then $X$ is \textbf{cellular} if $R^{2}\backslash X$ is
connected and simply connected (i.e. $X$ is a nonseparating planar
continuum).

The metric space $X$ is a \textbf{Peano continuum} if $X$ is compact,
connected, and locally path connected.

\section{Obtaining nested collections of PL disks $S_{n}$}

This section clarifies how we approximate a planar compactum $X$ (with
connected complement) by \ nested finite collections of pairwise disjoint PL
disks $S_{n}.$

\bigskip Recall definition \ref{MNdef}. Informally, using internal path
length distance on $S_{n}$, $M$ is the Hausdorff distance between $S_{n}$
and $S_{n+1}$ and $N$ is the Hausdorff distance between $S_{n+1}\cup
\partial S_{n}$ and $S_{n}.$

It is immediate $M\geq N$. To see why $M$ and $N$ can be dramatically
different, imagine that $S_{n}$ is a large round disk, and $S_{n+1}\subset
S_{n}$ is a continuum that approximates $\partial S_{n},$ but such that $%
S_{n+1}$ is far from the center of the disk $S_{n}.$

Despite the above disparity we have the following Lemma.

\begin{lemma}
\label{MN}Suppose $\forall n\geq 1,$ $S_{n}\subset R^{2}$ is the union of
finitely many pairwise disjoint PL closed topological disks such that $%
S_{n+1}\subset int(S_{n}).$ Then $\ \lim_{n\rightarrow \infty }$ $%
M(S_{n},S_{n+1})=0$ if and only if $N(S_{n},S_{n+1})=0$
\end{lemma}

\begin{proof}
Suppose $\lim_{n\rightarrow \infty }$ $M(S_{n},S_{n+1})=0.$ Then $%
\lim_{n\rightarrow \infty }$ $N(S_{n},S_{n+1})=0$ since $\forall n$ we have $%
N(S_{n},S_{n+1})\leq M(S_{n},S_{n+1}).$

Conversely suppose $\lim_{n\rightarrow \infty }$ $N(S_{n},S_{n+1})=0.$

Let $X=\cap _{n=1}^{\infty }S_{n}.$ Note if $x\in S_{n}$ and if $\varepsilon
>0$ and if $B(x,\varepsilon )\cap (S_{n+1}\cup \partial S_{n+1})\neq
\emptyset $ then there exists a PL path $\gamma \subset S_{n}$ from $x$ to $%
S_{n+1}$ whose length is less than $\varepsilon +N(S_{n},S_{n+1}).$

To obtain a contradiction suppose $\lim \sup M(S_{n},S_{n+1})>\varepsilon
>0. $ Then, retaining a subsequence and reindexing, there exists $x_{n}\in
S_{n}$ such that $B(x_{n},\varepsilon )\cap (S_{n+1}\cup \partial S_{n})=0.$

Note $B(x_{n},\varepsilon )\subset R^{2}\backslash X,$ since $X\subset
S_{n+1}.$

By compactness of $S_{1},$ ( once again retaining a subsequence and
reindexing) we may assume that $x_{n}\rightarrow x.$ Note $x\in S_{n}$ since 
$\{x_{n},x_{n+1},..\}\subset S_{n}$ and $S_{n}$ is closed. Thus $x\in X.$

On the other hand $\left| x-x_{n}\right| \rightarrow 0$. Thus for
sufficiently large $n,$ $x\in B(x_{n},\varepsilon )$ and thus $x\in
R^{2}\backslash X$ and we have a contradiction.
\end{proof}

\begin{theorem}
\label{Sn}Suppose $X\subset R^{2}$ is compact and $R^{2}\backslash X$ is
connected. Then there exists a sequence of closed sets $S_{n}\subset R^{2}$
such that $S_{n}$ is the union of finitely many pairwise disjoint closed PL
topological disks, such that $S_{n+1}\subset int(S_{n})$, such that $X=\cap
_{n=1}^{\infty }S_{n}$ and such that $N(S_{n},S_{n+1})<\frac{1}{10^{n}}$ and
such that 
\begin{equation*}
\lim_{n\rightarrow \infty }M(S_{n},S_{n+1})=0.
\end{equation*}
\end{theorem}

\begin{proof}
Let $U=R^{2}\backslash X.$ Fix $z\in U.$ Obtain nested path connected closed
sets $A_{2}\subset A_{3}...$ such that $U=\cup _{n=2}^{\infty }A_{n}$ as
follows.

Let $T_{n}$ be a tiling of the plane (such that $z$ is a corner of some
tile), by closed squares of sidelength $\frac{1}{2^{n}}$ parallel to the $x$
or $y$ axis$.$

Let $A_{n}$ be a maximal path connected set containing $z$ such that each
closed tile of $A_{n}$ is strictly contained in $U$. Since $U$ is open and
path connected the sets $A_{n}$ cover $U.$

Since $X$ is compact obtain $R>0$ such that $X\subset \lbrack -R,R]\times
\lbrack -R,R]$ and let $S_{1}=[-R,R]\times \lbrack -R,R]$ and let $%
A_{1}=\emptyset .$

Suppose $n\geq 1$ and $S_{n}$ has been defined and $S_{n}$ is a collection
of pairwise disjoint topological disks such that $X\subset S_{n}$ and such
that $S_{n}\cap A_{n}=\emptyset .$

For each component $P\subset S_{n}$ and each $x\in $ $P\cap X$ obtain $%
0<\delta _{x}^{n}<\frac{1}{10^{n}}$ such that $\overline{B(x,\delta _{x}^{n})%
}\subset int(P)$ and such that $\overline{B(x,\delta _{x}^{n})}\cap
A_{n+1}=\emptyset .$ By compactness of $P\cap X$ we obtain a finite
subcovering $\{B\{x_{i},\delta _{x_{i}}^{n}\}\}.$ Let $Y_{n+1}=\cup 
\overline{B(x_{i},\delta _{x_{i}}^{n})}$ $\ $with components $%
Q_{1},...Q_{m}. $ By definition for each $y\in Q_{i}$ there exists $x\in X$
such that the line segment $[y,x]\subset Q_{i}$ and such that $l([y,x])<%
\frac{1}{10^{n}}.$

Notice $Q_{i}$ is a locally contractible planar continuum and consequently $%
Q_{i}$ has the homotopy type of a disk with finitely many open punctures.

Thicken the outer boundary of $Q_{i}$ very slightly to obtain pairwise
disjoint PL disks $P_{1},..P_{m}$ such that $Q_{i}\subset int(P_{i})$ and
such that $N(P_{i},X\cap P_{i})<\frac{1}{10^{n}}.$

Let $S_{n+1}$denote the union of the PL disks $P_{i}$ obtained in the
fashion just described.

By construction $X\subset S_{n+1}\subset int(S_{n})$ and $S_{n+1}$ is the
union of finitely many pairwise disjoint $P$ disks such that if $P$ is a
component of $S_{n+1}$ then $X\cap P\subset int(P).$ To see that $X=\cap
_{n=1}^{\infty }S_{n}$ it is immediate that $X\subset \cap S_{n}$ since $%
X\subset S_{n}$ for each $n.$ Conversely suppose $y\notin X$. Then there
exists $n\geq 2$ such that $y\in A_{n}$ and in particular $y\notin S_{n}.$

Since $N(S_{n},S_{n+1})<\frac{1}{10^{n}},$ it follows that $%
N(S_{n},S_{n+1})\rightarrow 0$ and hence by Lemma \ref{MN} $%
M(S_{n},S_{n+1})\rightarrow 0.$
\end{proof}

\section{The $CAT(0)$ geometry of a PL planar disk.}

If $P\subset R^{2}$ is a closed PL disk, then internal paths (in $P$) of
minimal Euclidean length determines a metric as follows.

Define $d_{P}:P\times P\rightarrow \lbrack 0,\infty )$ so that $d(x,x)=0$
and $\forall M\geq 0,$ and $x\neq y,$ $d(x,y)\leq M$ iff there exists a PL
arc $\alpha \subset P$ such that $\partial \alpha =\{x,y\}$ and $l(\alpha
)\leq M.$

Note $d_{P}$ is a topologically compatible metric (since $P$ is locally path
connected and short paths exist locally).

Using polar coordinates, if $R>0$ and $0<\psi \leq 2\pi $ let $D(R,\psi
)=\{(r,\theta )\in R^{2}|r<R$ and $\theta <\psi $ $\}.$

(For a careful exposition of the elementary properties of $CAT(0)$ spaces we
refer the reader to \cite{Bridson}.)

Notice $P$ has a basis of open sets each of which is isometric to some set
of the form $D(R,\psi ),$ a round Euclidean disk, possibly missing an open
sector.

By inspection, pairs of points in $D(R,\psi )$ are connected by a canoniical
unique path of minimal length. Moreover the triangles $T\subset D(R,\psi )$
are `thin', a pair of points in $T$ is at least as close in $T$ as their
counterparts in the canonical Euclidean comparison triangle $S.$

Thus $D(R,\psi )$ is a \textbf{CAT(0) }space, and hence $(P,d_{P})$ is
locally $CAT(0).$ Since $P$ is compact and simply connected, $(P,d_{p})$ is $%
CAT(0)$.

There are 3 important properties of $(P,d_{P})$ (which are true in any a $%
CAT(0)$ space).

1) Given $\{x,y\}\subset P$ there exists a unique arc (or point if $x=y$) of
minimal length (a \textbf{geodesic }) connecting $x$ and $y.$

2) The geodesics and their lengths vary continuously with the endpoints.

3) The intersection of two geodesics is connected or empty.

However we will also need the following 4th (special) property of $(P,d_{P})$
concerning the geometry of an embedded triod $T\subset P.$

\begin{lemma}
\label{notriod}Suppose $P\subset R^{2}$ is a PL disk with $CAT(0)$ metric
determined by paths of minimal Euclidean length. Suppose $T\subset P$ is a
topological triod such that $\partial T=\{a,b,c\}$ and $x\in T$ is the
vertex. Suppose $\alpha _{ab},\alpha _{bc},\alpha _{ac}$ denote the arcs in $%
T$ connecting the respective pairs of endpoints. Then at least one of $%
\alpha _{ab},\alpha _{bc},\alpha _{ac}$ is not a geodesic.
\end{lemma}

\begin{proof}
The idea is to notice on the small scale near $x$, $T$ consists of 3
distinct straight segments emanating from $x.$ On the small scale at least
one of the 3 complementary sectors must be contained in $int(P)$, and this
creates the possibility to shorten the side of $T$ bounding the selected
sector.

To obtain a contradiction let $[a,x]\cup \lbrack x,b],$ and $[c,x]\cup
\lbrack x,b]$ and $[c,x]\cup \lbrack x,a]$ denote the geodesic sides of $T.$
In particular $T$ is convex in $P.$

On the other hand we may choose $\varepsilon >0$ so that if $t<\varepsilon $
we have 3 distinct Euclidean line segments emanating from $x\in P:$ $%
[x,ta]\subset \lbrack x,a],$ $[x,tb]\subset \lbrack x,b],$ and $%
[x,tc]\subset \lbrack x,c].$

Since $P\subset R^{2}$ is a PL disk, there exists $\delta <\varepsilon ,$
such that $\overline{B(x,\delta )}\cap P$ is isometric to a round disk with
an open (possibly empty) sector missing ( and we allow that a sector can
have angle $>180^{0}$). Note $\overline{B(x,\delta )}\cap P$ is convex in $%
P. $

However, by inspection $T\cap \overline{B(x,\delta )}$ is not convex in $%
\overline{B(x,\delta )}\cap P$ contradicting our assumption that $T$ is
convex in $P.$
\end{proof}

\section{Perturbing a PL disk}

Given a PL disk $Q\subset R^{2},$ with finitely many marked points $Y\subset
\partial Q$ we wish to show the existence of arbitrarily small perturbations 
$P$ of $Q,$ so that the respective geometries of $P$ and $Q$ are very close
to each other$,$ and so that $P\subset Q$ and $Q\cap \partial P=Y.$

This is obvious and the idea is simply to push the components of $\partial
Q\backslash Y$ inward by a very tiny amount.

\begin{lemma}
\label{ham}Suppose $\alpha \subset R^{2}$ is a PL arc and $\delta >0.$ There
exist PL arcs $\beta \subset R^{2}$ and $\gamma \subset R^{2}$ such that $%
int(\beta )\cap int(\gamma )=\emptyset $ and $\alpha $ is a spanning arc of
the PL disk $D(\beta ,\gamma )$ and there exists a homeomorphism $h:D(\beta
,\gamma )\rightarrow D(\beta ,\alpha )$ such that $\left| d_{D(\beta ,\alpha
)}(h(x),h(y))-d_{D(\beta ,\gamma )}(x,y)\right| <\delta .$
\end{lemma}

\begin{proof}
Let $v_{0},....,v_{n}$ denote consecutive vertices of $\alpha .$ Notice if $%
\gamma $ is sufficiently small there exist points $w_{1},...w_{n-1}$ such
that $w_{i}\notin \alpha $ and $\left| w_{i}-v_{i}\right| <\gamma $ and such
that all of the following hold:

1)\ The PL path $\beta =[w_{0},w_{1}]\cup \lbrack w_{1},w_{2}]...\cup
\lbrack w_{n-1},w_{n}]$ is an arc such that $int(\alpha )\cap int(\beta
)=\emptyset .$

2) If $T_{i}$ denotes the convex hull of $\{v_{i-1},w_{i-1},v_{i},w_{i}\},$
then $T_{i}$ is convex in $R^{2}$.

3) The closed disk $D(\beta ,\alpha )$ bounded by $\alpha $ and $\beta $ can
be canonically fibred by line segments as follows:

If $L_{i}:[v_{i},v_{i+1}]\rightarrow \lbrack w_{i},w_{i+1}]$ is the order
preserving linear homeomorphism then $[v_{i,t},L_{i}(v_{i,t})]\subset
D(\beta ,\alpha \dot{)}$

and if $j\neq i$ or if $s\neq t$ then $[v_{i,t},L_{i}(v_{i,t})]\cap \lbrack
v_{j,s},L_{i}(v_{j,s})]=\emptyset .$

4) If $i\in \{2,3,...n-1\}$ then $D(\beta ,\alpha )\backslash T_{i}$ is
disconnected.

Given $\alpha $ and $\beta $, we can perform a similar construction on the
other side of $\alpha $ to obtain an arc $\gamma $ with the same properties
(with respect to $\alpha $) as $\beta $ such that $int(\gamma )\cap
int(\beta )=\emptyset .$

Let $S_{1}<..<S_{n}$ denote the convex cells of $D(\alpha ,\gamma )$ and let 
$u_{0},..u_{n}$ denote the vertices of $\gamma .$

Now from the above construction we have canonical PL homeomorphisms $%
f:\alpha \rightarrow \beta $ and $g:\alpha \rightarrow \gamma .$

Hence we obtain a canonical homeomorphism $h:D(\beta ,\gamma )\rightarrow
D(\alpha ,\gamma )$ as follows. For each $x\in \alpha ,$ $\ $let $h$ map the
PL arc $[f(x),x]\cup \lbrack x,g(x)]$ `linearly' onto the segment $[x,g(x)]$
so that $h_{[x,g(x)]}^{-1}$ has constant speed.

Q Since $T_{i}$ is convex, if the line segment $\gamma _{i}\subset T_{i}$
connects $[w_{i-1},v_{i-1}]$ to $[v_{i},w_{i},]$ then $l([w_{i-1},v_{i-1}])%
\leq l(\gamma _{i})\leq l([v_{i},w_{i},])$ or $l([w_{i-1},v_{i-1}])\geq
l(\gamma _{i})\geq l([v_{i},w_{i},]).$

If $\gamma $ is small then $\left| \left| w_{i-1}-w_{i}\right| -\left|
v_{i-1}-v_{i}\right| \right| <\frac{\delta }{2n}$ and $\left| \left|
u_{i-1}-u_{i}\right| -\left| v_{i-1}-v_{i}\right| \right| <\frac{\delta }{2n}
$

By condition 4, a given geodesic $\lambda \subset D(\beta ,\gamma \dot{)}$
is the union of line segments from consecutive cells $T_{i}\cup
S_{i},...,T_{i+k}\cup S_{i+k}.$

Consequently $l(h(\gamma ))<l(\lambda )+n(\frac{\delta }{2n}+\frac{\delta }{%
2n})=l(\lambda )+\delta .$

In similar fashion a given geodesic $\lambda \subset D(\alpha ,\gamma \dot{)}
$ is the union of line segments from consecutive cells $%
S_{i},S_{i+1},...S_{i+k}.$ Hence $l(h^{-1}(\lambda )<l(\lambda )+\delta .$
\end{proof}

\begin{lemma}
\label{push}Suppose $P\subset R^{2}$ is a PL disk and $Y=\{y_{1}..,y_{n}\}%
\subset \partial P$ and $\varepsilon >0.$ There exists a PL disk $Q\subset
R^{2}$ such that $P\subset Q$ and $\partial P\cap \partial Q=Y$ and there
exists a homeomorphism $h:Q\rightarrow P$ such that $\left|
d_{Q}(x,y)-d_{P}(h(x),h(y))\right| <\varepsilon $ and such that $%
h_{Y}=id_{Y}.$
\end{lemma}

\begin{proof}
Let $\alpha _{1},...\alpha _{n}$ denote the arcs between consecutive points
of $Y$ so that $\cup \alpha _{i}=\partial P.$ Let $A$ denote the set of
vertices of $\partial P$ and let $M=\left| A\cup Y\right| .$

Apply the previous Lemma with $\delta =\frac{\varepsilon }{M}$ to each $%
\alpha _{i}$ so that $int(D(\beta _{i},\gamma _{i}))\cap int(D(\beta
_{j},\gamma _{j}))=\emptyset $ if $i\neq j.$

Let $Q=\cup D(\beta _{i},\alpha _{i})\cup P$ and let $h_{i}:D(\beta
_{i},\gamma _{i})\rightarrow D(\alpha _{i},\gamma _{i})$ be the canonical
homeomorphism (defined in Lemma \ref{ham}) and let $E$ denote the PL disk
bounded by $\cup \gamma _{i}.$

Notice we have a total of $2M$ sets of the form $T_{j}^{i}$ or $S_{j}^{i}$
and each set $T_{j}^{i}$ and each set $S_{j}^{i}$ is a convex Euclidean
quadrilateral.

In particular if $\lambda \subset Q$ is a geodesic then $\gamma \cap
T_{j}^{i}$ is a (connected) geodesic and $\gamma \cap S_{j}^{i}$ is a
(connected) geodesic and thus $l(h(\gamma \cap T_{j}^{i}))<l(\gamma \cap
T_{j}^{i})+\frac{\varepsilon }{2M}$ and $l(h(\gamma \cap
S_{j}^{i}))<l(\gamma \cap S_{j}^{i})+\frac{\varepsilon }{2M}$

In similar fashion if $\lambda \subset P$ is a geodesic then $\lambda \cap
S_{j}^{i}$ is a (connected) geodesic and $l(h^{-1}(\gamma \cap
S_{j}^{i}))<l(\gamma \cap S_{j}^{i})+\frac{\varepsilon }{2M}.$

Thus if $\lambda \subset Q$ is a geodesic we let $\lambda =\psi _{1}\cup
\psi _{2}....$ such that $\psi _{1},\psi _{2},...$ are consecutive distinct
components of either $\lambda \cap E$ or $\lambda \cap (\cup _{i}D(\beta
_{i},\gamma _{i})).$

Note if $\psi _{j}$ is a component of $\lambda \cap E$ then $h(\psi
_{j})=\psi _{j}$ and hence $l(\psi _{j})=l(h(\psi _{j}))$

If $\psi $ is a component of $\lambda \cap D(\beta _{i},\gamma _{i})$ then $%
l(h(\psi _{j}))<l(\psi _{j})+\frac{\varepsilon }{M}$ and thus $l(h(\lambda
))=\Sigma l(h(\psi _{i}))+\varepsilon $

In similar fashion if $\lambda \subset P$ is a geodesic we let $\lambda
=\psi _{1}\cup \psi _{2}....$ such that $\psi _{1},\psi _{2},...$ are
consecutive distinct components of either $\lambda \cap E$ or $\lambda \cap
(\cup _{i}D(\alpha _{i},\gamma _{i})).$

Note if $\psi _{j}$ is a component of $\lambda \cap E$ then $h^{-1}(\psi
_{j})=\psi _{j}$ and hence $l(\psi _{j})=l(h^{-1}(\psi _{j}))$

If $\psi _{j}$ is a component of $\lambda \cap D(\alpha _{i},\gamma _{i})$
then $l(h^{-1}(\psi _{j}))<l(\psi _{j})+\frac{\varepsilon }{M}$ and thus $%
l(h^{-1}(\lambda ))=\Sigma l(h(\psi _{i}))+\varepsilon .$
\end{proof}

\begin{theorem}
\label{perturb}Suppose $Q\subset R^{2}$ is a PL disk and $Y\subset \partial
Q $ is finite. Suppose $E\subset int(Q)\cup Y$ and assume $E$ is a $PL$
compactum with finitely many components, and assume $\varepsilon >0$. Then
there exists a PL disk $P\subset Q$ such that $P\cap \partial Q=Y$ and there
exists a homeomorphism $h:Q\rightarrow P$ such that if $d_{P}$ and $d_{Q}$
denote the respective $CAT(0)$ metrics on $P$ and $Q$ then $h_{E}=id_{E},$
and if $\{x,y\}\subset Q$ then $\left| d_{P}(h(x),h(y))-d_{Q}(x,y)\right|
<\varepsilon .$
\end{theorem}

\begin{proof}
Starting with $Q$, push the components of $\partial Q\backslash Y$ inward by
a sufficiently tiny amount, obtaining a PL disk $P\subset Q$ so that so
that, recalling Lemma \ref{push}, the homeomorphism $h:Q\rightarrow P$ fixes 
$E\cup Y$ pointwise.
\end{proof}

\section{Connecting $\partial S_{n}$ to $S_{n}$ with very short arcs}

Recall we have $S_{n+1}\subset S_{n}$ and $S_{i}$ is a collection of
pairwise disjoint PL disks in $R^{2}$ such that $S_{n+1}\subset int(S_{n})$
and each component of $S_{n}$ contains at least one component of $S_{n+1}.$

Given finitely many marked points $Y_{n}\subset \partial S_{n},$ we wish
naively to connect $Y_{n}$ to $\partial S_{n+1}$ with pairwise disjoint arcs
in $S_{n}$ of minimal length. The simplest strategy almost works. If we
select the shortest possible arcs while ignoring the constraint of
disjointedness, the $CAT(0)$ geometry guarantees the selected arcs don't
cross transversely. Consequently we can perturb (arbitrarily) the selected
arcs to be disjoint, (and in our particular application all will have length
less than $\frac{1}{10^{n}}$).

\begin{lemma}
\label{treelem}Suppose $\alpha _{1},..\alpha _{n}$ are PL arcs in the plane
with common endpoint $z$ such that $\alpha _{i}\cap \alpha _{j}$ is
connected. Then $\cup _{i=1}^{n}\alpha _{i}$ is a tree.
\end{lemma}

\begin{proof}
Note $\alpha _{1}$ is simply connected. By induction suppose $\cup
_{i=1}^{k-1}\alpha _{i}$ is simply connected. Let $\partial \alpha
_{k}=\{x,z\}.$ If $x\in \cup _{i=1}^{k-1}\alpha _{i}$ then let $x\in \alpha
_{i}$ for $i\leq k-1.$ Thus $\alpha _{k}\subset \alpha _{i}$ since, by the
hypothesis of the Lemma, $\alpha _{k}\cap \alpha _{i}$ is connected. Hence $%
\cup _{i=1}^{k-1}\alpha _{i}=\cup _{i=1}^{k}\alpha _{i}$ and by the
induction hypothesis both spaces are simply connected.

If $x\notin \cup _{i=1}^{k-1}\alpha _{i}$ let $J$ be the component of $(\cup
_{i=1}^{k-1}\alpha _{i})\backslash \alpha _{k}$ such that $x\in J.$ Let $y=%
\overline{J}\backslash J.$ Let $y\in \alpha _{i}$ for $i<k.$ Then $(\alpha
_{k}\backslash J)\subset \alpha _{i}$ since $\{y,z\}\subset \alpha _{k}\cap
\alpha _{i}.$

To obtain a contradiction suppose $\cup _{i=1}^{k}\alpha _{i}$ is not simply
connected. Let $S\subset \cup _{i=1}^{k}\alpha _{i}$ be a simple closed
curve. By the induction hypothesis $\cup _{i=1}^{k-1}\alpha _{i}$ is simply
connected and hence $\alpha _{k}\cap S\neq \emptyset $ and hence $S\cap
J\neq \emptyset .$ Now we have a contradiction since $J$ is an initial
segment of $\alpha _{k}$ and $J\cap (\cup _{i=1}^{k-1}\alpha _{i})=\emptyset
.$
\end{proof}

\begin{theorem}
\label{shrtarcs}Suppose each of the sets $S_{n}\subset R^{2}$ and $%
S_{n+1}\subset R^{2}$ is a collection finitely many pairwise disjoint closed
PL disks. Suppose $S_{n+1}\subset int(S_{n})$ and $N(S_{n},S_{n+1})<\delta
_{n}.$ Suppose $Y_{n}=\{y_{1},,.,y_{m}\}\subset \partial S_{n}$. Then there
exists a collection of pairwise disjoint closed arcs $\gamma _{1},..\gamma
_{m}$ such that $int(\gamma _{i})\subset int(S_{n})\backslash S_{n+1}$ and $%
\gamma _{i}$ connects $y_{i}$ to $S_{n+1}$ and $l(\gamma _{i})<\delta _{n}.$
\end{theorem}

\begin{proof}
The strategy is to first recursively select arcs of minimal length $\alpha
_{1},..\alpha _{m}$ in $S_{n}$ such that $\alpha _{i}$ connects $y_{i}$ to $%
S_{n+1}.$ Unfortunately it can happen that $\alpha _{i}\cap \alpha _{j}\neq
\emptyset .$ However, by virtue of our construction, if $\alpha _{i}\cap
\alpha _{j}\neq \emptyset $ then $\alpha _{i}\cap \alpha _{j}$ is a final
segment of each of $\alpha _{i}$ and $\alpha _{j}$. Consequently we can
perturb the arcs $\{\alpha _{i}\}$ to very short pairwise disjoint arcs $%
\{\gamma _{i}\}.$

Recall each component of $S_{n}$ admits a canonical $CAT(0)$ metric
determined by minimal Euclidean path length between points.

By compactness of $S_{n},$ there exists an arc $\alpha _{1}\subset S_{n}$ of
minimal length connecting $y_{1}$ to $\partial S_{n+1}.$ Let $\partial
(\alpha _{1})=\{y_{1},z_{1}\}$. We proceed recursively as follows.

Suppose the geodesic arcs $\alpha _{1},,\alpha _{i-1}\subset S_{n}$ have
been chosen so that if $k\leq i-1$ then $\alpha _{k}$ is a path of minimal
length connecting $y_{k}$ to $\partial S_{n+1}$ at $z_{k}$ and suppose if $%
k<j\leq i-1$ and if $\alpha _{k}\cap \alpha _{j}\neq \emptyset $ then $%
z_{k}=z_{j}.$

Let $\alpha _{i}^{\ast }$ be a minimal arc connecting $y_{i}$ to $\partial
S_{n+1}$ in $S_{n}.$ If $\alpha _{i}^{\ast }\cap \alpha _{j}=\emptyset $ for
all $j<i$ then let $\alpha _{i}=\alpha _{i}^{\ast }$ and notice the
induction hypothesis holds for $\{\alpha _{1},..,\alpha _{i}\}.$

If there exists $j<i$ such that $\alpha _{i}^{\ast }\cap \alpha _{j}\neq
\emptyset $ let $\alpha _{i}^{\ast }\cap \alpha _{j}=[x,y]$ with $y_{j}\leq
x\leq y\leq z_{j}.$ Notice $l([x,z_{i}])=l([x,z_{j}])$ since otherwise one
of $\alpha _{j}$ or $\alpha _{i}$ could be strictly shortened contradicting
minimality. Define $\alpha _{i}=[y_{i},x]\cup \lbrack x,z_{j}].$ Note $%
l(\alpha _{i}^{\ast })=l(\alpha _{i})$ and hence $\alpha _{i}$ has minimal
length. If $\alpha _{i}\cap \alpha _{k}\neq \emptyset $ then, by definition
of $\alpha _{i}$ and the induction hypothesis we have $z_{i}=z_{j}=z_{k}.$

Finally, note since $N(S_{n},S_{n+1})<\delta _{n},$ that $l(\alpha
_{i})<\delta _{n}$ for all $i.$

To perturb the arcs $\{\alpha _{i}\},$ first notice the collection $\{\alpha
_{i}\}$ is naturally partitioned via the equivalence relation $\alpha _{i}%
\symbol{126}\alpha _{j}$ if and only if $z_{i}=z_{k}.$

Thus for each equivalence class $[z_{i}]$ consider the arcs $\beta
_{1}^{i},...,\beta _{k}^{i}\subset \{\alpha _{1},,,\alpha _{m}\}$ such that $%
\beta _{j}^{i}$ connects $y_{j}^{i}$ to $z_{i}.$ Note if $z_{k}\neq z_{i}$
then $(\cup _{j}\beta _{j}^{i})\cap (\cup _{j}\beta _{j}^{k})=\emptyset $
since if $z_{i}\neq z_{k}$ then $\alpha _{i}\cap \alpha _{k}=\emptyset .$

Fixing $i,$ note $\{y_{j}^{i}\}$ inherits a canonical circular order from
the simple closed curve component of $\partial S_{n}.$ Now we will select a
`starting point' $y\in \{y_{j}^{i}\}$ as follows.

Let $D_{i}$ denote the component of $S_{n+1}$ such that $z_{i}\in D_{i}.$
Notice $z_{i}\in \partial D_{i}$ since $\beta _{j}^{i}$ has minimal length.

Let $T^{i}=\cup \beta _{j}^{i}.$ By Lemma \ref{treelem} $T^{i}$ is a finite
PL tree, and, by hypothesis, if $k\neq i$ then $T^{i}\cap T^{k}=\emptyset .$

Fix an arbitrary point $w\in \partial D_{i}\backslash z_{i}.$ Start at $w$
and travel clockwise along $\partial D_{i}$ and stop at $z_{i}.$ Then travel
monotonically along $T^{i},$ turning `left' whenever possible and stop at $%
y\in Y_{n},$ and we have canonically obtained a `starting point' \ from our
circularly ordered set $\{y_{j}^{i}\}$.

Now, keeping $i$ fixed and permuting $j,$ reindex $\{\beta _{j}^{i}\}$such
that $y=y_{1}^{i}<...<y_{m}^{i}$, ordered in clockwise fashion on $\partial
S_{n}.$

Theorem \ref{perturb} ensures we can obtain arbitrarily small perturbations
of $\{\beta _{j}^{i}\}$ to obtained the desired arcs $\{\gamma _{j}^{i}\}.$
\end{proof}

\section{\label{alg}Connecting cellular sets in PL disks with short arcs}

Recall at the $nth$ stage of our construction we have $S_{n+1}\subset S_{n}$
and $S_{i}$ is a collection of pairwise disjoint PL disks such that $%
S_{n+1}\subset int(S_{n})$ and each component of $S_{n}$ contains at least
one component of $S_{n+1}.$ At this stage we also have finitely many arcs
pairwise disjoint PL arcs $\gamma _{1},\gamma _{2},...$ connecting $\partial
S_{n}$ to $S_{n+1}$ such that $l(\gamma _{i})<\frac{1}{10^{n}}$ and such
that $int(\gamma _{i})\subset S_{n}\backslash S_{n+1}.$

Note each component $D_{j}\subset S_{n+1}\cup \gamma _{1}\cup \gamma _{2}...$
is a PL cellular set.

Our naive hope is to attach closed disjoint arcs $\alpha _{1},\alpha _{2},,,$%
of minimal length to $\cup D_{j}$ in order create one cellular continuum in
each component of $S_{n},$ and we also hope that $\alpha _{i}\cap \gamma
_{j}=\emptyset .$

The simplest strategy almost works. If we begin connecting together the sets 
$D_{1},D_{2},..$ with minimal length arcs $\alpha _{1},\alpha _{2},..$,
without regard to whether the newly selected arcs $\{\alpha _{n}\}$ are
disjoint, ultimately the $CAT(0)$ structure on the components of $S_{n}$
ensures our newly selected arcs do not cross transversely.

However there are two technical problems with the output arcs $\alpha
_{1},\alpha _{2},..$

As mentioned, the first problem is the arcs $\{\alpha _{n}\}$ are typically
not disjoint, however this can be fixed by arbitrarily small perturbations,
since the arcs $\{\alpha _{n}\}$ do not cross each other transversely.

The second problem is that $\alpha _{i}\cap \gamma _{j}\neq \emptyset $ can
happen, yet we were hoping that $\alpha _{i}\cap \gamma _{j}=\emptyset $ for
all $i$ and $j.$ This problem can be fixed by a perturbation on the order of 
$\frac{1}{10^{n}},$ since $l(\gamma _{i})<\frac{1}{10^{n}},$ and the idea is
to `slide' the arcs $\alpha _{i}$ off of $\gamma _{j}$ (in a tiny
neighborhood of $\gamma _{j}$) so that $\alpha _{i}\subset S_{n}$ connects
distinct PL disks from $S_{n+1}.$

By construction the originally selected arcs $\alpha _{1},\alpha _{2},...$%
are `short' and have length at most $2M(S_{n},S_{n+1})$ ( double the
Hausdorff (using path length) distance between $S_{n},S_{n+1}$).

After two or three perturbations of the arcs $\{\alpha _{n}\}$ we have arcs $%
\{\alpha _{n}^{\ast \ast \ast }\}$ such that $l(\alpha _{n}^{\ast \ast \ast
})<2M(S_{n},S_{n+1})+2N(S_{n},S_{n+1})$ with all the desired properties,
namely $\alpha _{n}^{\ast \ast \ast }$ connects distinct components of $%
S_{n+1}$ and $int(\alpha _{n}^{\ast \ast \ast })\subset int(S_{n}),$ and $%
\alpha _{i}^{\ast \ast \ast }\cap \gamma _{j}=\emptyset $ for all $i$ and $%
j. $

Combining all the constructions and perturbations in this section \ref{alg}
(and its subsections), we obtain the following theorem, ultimately critical
to the recursive process by which we will attach a null sequence of arcs to
an arbitrary planar compactum.

\begin{theorem}
\label{main3}Suppose $P\subset R^{2}$ is a closed PL disk and $%
E_{1},E_{2},..,E_{n}$ are pairwise disjoint closed PL disks such that $\cup
E_{i}\subset int(P)$. Suppose $\{y_{1},...,y_{m}\}\subset P$ and suppose $%
\gamma _{1},\gamma _{2},..,\gamma _{m}$ is a collection of pairwise disjoint
closed PL arcs such that $\gamma _{i}$ connects $y_{i}$ to $\partial (\cup
E_{i})$ and such that $l(\gamma _{i})<N(P,\cup E_{i})$ and such that $%
int(\gamma _{i})\subset P\backslash (\cup E_{i}).$ Then there exist finitely
many pairwise disjoint PL closed arcs $\alpha _{1}^{\ast \ast \ast },\alpha
_{2}^{\ast \ast \ast },...$ $\subset P,$ such that $l(\alpha _{i}^{\ast \ast
\ast })<2(M(P,\cup E_{i})+N(P,\cup E_{i}))$ and such that $\alpha _{i}^{\ast
\ast \ast }$ connects distinct components of $\cup E_{i},$ such that $%
int(\alpha _{i}^{\ast \ast \ast })\subset P\backslash ((\cup E_{i})\cup
(\cup \gamma _{j})),$ and such that $\{\cup \alpha _{i}^{\ast \ast \ast
}\}\cup \{\cup \gamma _{j}\}\cup \{\cup E_{k}\}$ is cellular.
\end{theorem}

\subsection{\label{arcselect}An arc selection algorithm}

The \textbf{input} for the algorithm is the data $P,D_{1},..D_{n}$
satisfying the following two conditions:

1) $P\subset R^{2}$ is a closed PL disk.

2) $\{D_{i}\}$ is a collection of pairwise disjoint cellular sets such that $%
\cup _{i=1}^{n}D_{i}\subset P$.

Consider the topologically compatible $CAT(0)$ metric $d:P\times
P\rightarrow \lbrack 0,\infty )$ satisfying $\forall M\geq 0,$ $d(x,y)\leq M$
iff there exists a PL path in $P$ connecting $x$ to $y$ of Euclidean
pathlength $M$ or less.

Starting at $k=1$ select arcs $\alpha _{k}\subset P$ recursively as follows.

Let $F_{k}=D_{1}\cup D_{2}..\cup D_{n}\cup \alpha _{1}...\alpha _{k-1}.$

If $F_{k}$ is connected terminate the algorithm.

If $F_{k}$ is not connected, let $C_{k}$ denote the set of all paths $\beta $
in $P$ such that $\beta $ connects distinct components of $F_{k}$ and such
that the endpoints of $\beta $ belong to distinct components of $\cup
_{i=1}^{n}D_{i}$.

Let $P_{k}$ denote the set of all $g\in C_{k}$ such that $g$ has minimal
length. Note $P_{k}\neq \emptyset $ since $F_{k}$ and $\cup _{j=1}^{n}D_{j}$
are compact.

If possible, select $\alpha _{k}\in P_{k}$ such that for all $i<k,$ $\alpha
_{i}$ $\cap \alpha _{k}$ does not disconnect $\alpha _{k}.$ Otherwise
terminate the algorithm.

The algorithm eventually terminates since $F_{0}$ has finitely many
components and $F_{k-1}$ has strictly fewer components than $F_{k}.$

\begin{lemma}
Suppose $\alpha _{1},..\alpha _{k}$ have been selected by the previous
algorithm. Suppose $M(P,\cup D_{i})<\varepsilon .$ Then for each $i,$ $%
l(\alpha _{i})<2\varepsilon $ and $int(\alpha _{i})\subset P\backslash (\cup
_{k=1}^{n}D_{k}).$ If $i<k$ then $l(\alpha _{i})\leq l(\alpha _{k}).$
\end{lemma}

\begin{proof}
Suppose $A$ and $B$ is any separation of $\cup _{j=1}^{n}D_{j}$. Since $A$
and $B$ are compact, there exists a path of minimal length $g$ connecting $A$
to $B$ and such a path must be a geodesic arc (since all nongeodesics in $P$
can be strictly shortened while keeping the endpoints fixed), and since all
nontrivial geodesics in a $CAT(0)$ space are topological arcs.

Let $x$ be the midpoint of $g$ and let $g=[a,x]\cup \lbrack x,b]$ with $%
l([a,x])=l([x,b])$ and $a\in A$ and $b\in B.$

Let $[x,z]$ be a geodesic connecting $x$ to $A\cup B$ such that $%
l([x,z])<\varepsilon .$ Since $z\in A\cup B$ wolog we may assume $z\in A.$
To obtain a contradiction assume $l([z,x])<l([a,x]).$ Then the path $%
[z,x]\cup \lbrack x,b]$ connects $A$ to $B$ and $l([z,x]\cup \lbrack
x,b])<l(g)$ contradicting the fact that $g$ has minimal length among all
such paths. Thus $l([z,x])=l([a,x])$ and hence $l(g)<2\varepsilon .$

By definition $\alpha _{i}$ connects distinct components $E$ and $G$ of $%
F_{i}.$ Let $A=E\cap (\cup _{k=1}^{n}D_{k})$ and let $B=(F_{i}\backslash
E)\cap (\cup _{k=1}^{n}D_{k}).$ It follows that $\alpha _{i}$ is an arc of
minimal length connecting $A$ and $B.$ Thus $l(\alpha _{i})<2\varepsilon .$

Let $\beta _{i}:[0,1]\rightarrow \alpha _{i}$ be a homeomorphism such that $%
\beta _{i}(0)\in A.$

To prove $int(\alpha _{i})\subset P\backslash (\cup D_{i})$ , let $t$ be
maximal such that $\beta _{i}(t)\in A.$ It follows that there exists $\delta
>0$ such that if $s\in (t,t+\delta )$ then $\beta _{i}(s)\in \alpha _{i}\cap
(P\backslash (A\cup B)).$

Let $J$ the component of $\alpha _{i}\cap (P\backslash (A\cup B))$ such that 
$\beta _{i}(t,t+\delta ))\subset J.$ It follows that the other endpoint of $%
J $ is in $B.$

Then $J=int(\alpha _{i}),$ (since otherwise $l(J)<l(\alpha _{i})$
contradicting the fact that $l(\alpha _{i})$ is minimal among arcs in $P$
connecting $A$ and $B$.

Notice if $i<k$ then $C_{k}\subset C_{i}$ and hence $l(\alpha _{i})\leq
l(\alpha _{k}).$
\end{proof}

\subsubsection{If $k\leq n-1$ then $\protect\alpha _{k}$ exists.}

It is not obvious when our algorithm terminates. In principle two selected
arcs $\alpha _{i}$ and $\alpha _{j}$ could cross transversely and thus
terminate the algorithm prematurely.

However, the Theorem in this section shows the aforementioned disaster does
not occur.

In the proof we break the hypothetical data into cases and we argue the
least obvious case first, and we exploit symmetry of the data to cut down
the number of cases to a manageable size.

Lemma \ref{notriod} is of particular importance in the least obvious cases.

\begin{theorem}
For each $k<n$ there exists $\alpha _{k}\in P_{k}$ such that for all $i<k$ $%
\alpha _{i}\cap \alpha _{k}$ does not disconnect $\alpha _{k}.$
\end{theorem}

\begin{proof}
\textbf{\ }If $n=1$ the theorem is vacuously true. Suppose $n\geq 2.$ Notice 
$\alpha _{1}$ exists.

To obtain a contradiction assume the Theorem at hand is false. Choose $k<n-1$
minimal so that the theorem false.

Thus for all $\beta \in P_{k}$ there exists $j<k$ such that $\alpha _{j}\cap
\beta $ disconnects $\beta .$

Choose $i$ maximal so that there exists $\beta \in P_{k}$ such that $\beta
\cap \alpha _{j}$ does not disconnect $\beta $ for all $j<i.$

Now define $\alpha _{k}\in P_{k}$ such that $\alpha _{k}\cap \alpha _{j}$
does not disconnect $\alpha _{k}$ for all $j<i$ and such that $\alpha
_{k}\cap \alpha _{i}$ disconnects $\alpha _{k}.$

Note $\alpha _{i}\cap \alpha _{k}\neq \emptyset $ (since otherwise $\alpha
_{i}$ and $\alpha _{k}$ are disjoint and in particular $\alpha _{i}\cap
\alpha _{k}$ would fail to disconnect $\alpha _{k}.)$

Since $i\neq k,$ $\alpha _{i}\neq \alpha _{k}$ (since $\alpha _{k}$ connects
distinct components of $F_{k},$ and the continuum $\alpha _{i}$ is contained
in some component of $F_{k}$). Hence $\partial \alpha _{i}\neq \partial
\alpha _{k}$ by uniqueness of geodesics with common endpoints.

Thus $3\leq \left| \partial \alpha _{i}\cup \partial \alpha _{k}\right| \leq
4.$ We show $\left| \partial \alpha _{i}\cup \partial \alpha _{k}\right| =4$
as follows.

(If $3=\left| \partial \alpha _{i}\cup \partial \alpha _{k}\right| $ let $%
\{b\}=\partial \alpha _{i}\cap \partial \alpha _{k}.$ Thus $b\in \alpha
_{i}\cap \alpha _{k}.$ Moreover $\alpha _{i}\cap \alpha _{k}$ is a geodesic
since each of $\alpha _{i}$ and $\alpha _{k}$ is a geodesic. In particular $%
\alpha _{i}\cap \alpha _{k}$ is connected. Thus, since $b\in \alpha _{i}\cap
\alpha _{k},$ $\alpha _{i}\cap \alpha _{k}$ is an initial segment of $\alpha
_{k},$ contradicting our assumption that $\alpha _{k}\cap \alpha _{i}$
disconnects $\alpha _{k}.$)

Since $\left| \partial \alpha _{i}\cup \partial \alpha _{k}\right| =4,$ $%
(\alpha _{i}\cap \alpha _{k})\subset int(\alpha _{k})$ and $(\alpha _{i}\cap
\alpha _{k})\subset int(\alpha _{i}).$

Let $[z,y]=int(\alpha _{k})\cap int(\alpha _{i})$ and let $\alpha
_{i}=[a,z]\cup \lbrack z,y]\cup \lbrack y,b]$ and let $\alpha _{k}=[c,z]\cup
\lbrack z,y]\cup \lbrack y,d].$

Let $x\in \lbrack z,y].$Thus we have five distinct points $\{a,b,c,d,x\}.$
Moreover we see that each of $\{a,x,b,c\}$ is $\{a,x,b,d\}$ noncolinear as
follows.

(By symmetry it suffices to see that $c\cup \alpha _{i}$ is not colinear.
Let $\alpha $ be a geodesic containing $\alpha _{i}.$ Recall $c\notin
\lbrack a,b]$ and thus if $c\in \alpha _{i}$ then wolog $c<a<x$ on $\alpha .$
However $a\notin \lbrack c,x]$ and we have a contradiction.).

We begin with the hardest cases.

\textbf{A: The cases }$l[a,x]=l[b,x]$ $=l[c,x]$ or $l[a,x]=l[b,x]$ $=l[d,x]$

By symmetry it suffices to treat the case $l[a,x]=l[b,x]$ $=l[c,x].$

Let $G$ and $H$ denote the components of $F_{i}$ such that $a\in G$ and $%
b\in H.$

\textbf{Case A1. }If $c\notin G\cup H$ $\ $then $\{[a,x]\cup \lbrack
x,c],[c,x]\cup \lbrack x,b]\}\subset C_{i}$ and each have length equal to
that of $\alpha _{i}.$

By definition $[a,x]\cup \lbrack x,b]$ is a geodesic. By Lemma \ref{notriod}
, at least one of $[a,x]\cup \lbrack x,c]$ of $[c,x]\cup \lbrack x,b]$ is
not a geodesic and can be strictly shortened while keeping the endpoints
fixed, contradicting our choice of $\alpha _{i}.$

\textbf{Case A2. }Suppose $c\in G.$ Then $c\notin H.$ Let $\alpha _{k}^{%
\symbol{94}}=[a,x]\cup \lbrack x,d].$ Note $\alpha _{k}^{\symbol{94}}\in
C_{k}$ and $l(\alpha _{k}^{\symbol{94}})=l(\alpha _{k}).$ If $\alpha _{k}^{%
\symbol{94}}$ is not a geodesic then it can be strictly shortened while
keeping the endpoints fixed, contradicting our choice of $\alpha _{k}.$ Thus
we may assume that $\alpha _{k}^{\symbol{94}}$ is also a geodesic.

Suppose $j<i.$

\textbf{Case A2a. }Suppose $x\notin \alpha _{j}\cap \alpha _{k}^{\symbol{94}%
}.$

Then $(\alpha _{k}^{\symbol{94}}\cap \alpha _{j})\subset \lbrack a,x]$ or $%
(\alpha _{k}^{\symbol{94}}\cap \alpha _{j})\subset \lbrack x,d].$

If $(\alpha _{k}^{\symbol{94}}\cap \alpha _{j})\subset \lbrack a,x]$ then $%
\alpha _{k}^{\symbol{94}}\cap \alpha _{j}=\alpha _{i}\cap \alpha _{j}$ which
is an initial segment of $\alpha _{i}=[a,x]\cup \lbrack x,b]$ by induction
hypothesis. Hence $a\in (\alpha _{k}^{\symbol{94}}\cap \alpha _{j})$ and in
particular $(\alpha _{k}^{\symbol{94}}\cap \alpha _{j})$ does not separate $%
\alpha _{k}^{\symbol{94}}.$

If $(\alpha _{k}^{\symbol{94}}\cap \alpha _{j})\subset \lbrack d,x]$ then $%
\alpha _{k}^{\symbol{94}}\cap \alpha _{j}=\alpha _{k}\cap \alpha _{j}$ which
is an initial segment of $\alpha _{k}=[c,x]\cup \lbrack x,d]$ by the
induction hypothesis. Hence $d\in (\alpha _{k}^{\symbol{94}}\cap \alpha
_{j}) $ and in particular $(\alpha _{k}^{\symbol{94}}\cap \alpha _{j})$ does
not separate $\alpha _{k}^{\symbol{94}}.$

\bigskip \textbf{Case A2b. }Suppose $x\in \alpha _{j}\cap \alpha _{k}^{%
\symbol{94}}.$ Note $x\in \alpha _{i}\cap \alpha _{k}\cap \alpha _{j}.$

Since $int(\alpha _{j})\cap int(\alpha _{i})\neq \emptyset ,$ by induction
hypothesis $\alpha _{j}\cap \alpha _{i}$ is an initial segment of $\alpha
_{i}$.

If $a\in \alpha _{j}\cap \alpha _{i}$ then it follows that it follows that $%
\alpha _{k}^{\symbol{94}}\cap \alpha _{j}$ is an initial segment of $\alpha
_{k}^{\symbol{94}}$ and hence $\alpha _{k}^{\symbol{94}}\cap \alpha _{j}$
does not disconnect $\alpha _{k}^{\symbol{94}}.$

Thus we may assume henceforth that $b\in \alpha _{j}\cap \alpha _{i}.$ It
follows that $[b,x]\subset (\alpha _{j}\cap \alpha _{i})$ since $\alpha
_{j}\cap \alpha _{i}$ is a geodesic containing $\{x,b\}$

Since $int(\alpha _{j})\cap int(\alpha _{k})\neq \emptyset ,$ by induction
hypothesis $\alpha _{j}\cap \alpha _{k}$ is an initial segment of $\alpha
_{k}.$

If $c\in \alpha _{j}\cap \alpha _{k}$ then it follows that $\alpha _{k}^{%
\symbol{94}}\cap \alpha _{j}$ is an initial segment of $\alpha _{k}^{\symbol{%
94}}$ and hence $\alpha _{k}^{\symbol{94}}\cap \alpha _{j}$ does not
disconnect $\alpha _{k}^{\symbol{94}}.$

Thus we may assume henceforth that $d\in \alpha _{j}\cap \alpha _{k}.$ It
follows that $[d,x]\subset (\alpha _{j}\cap \alpha _{k})$ since $\alpha
_{j}\cap \alpha _{k}$ is a geodesic containing $\{x,d\}$

By Lemma \ref{notriod}, since $[a,x]\cup \lbrack x,d]$ and $[a,x]\cup
\lbrack x,b]$ are geodesics, the path $[b,x]\cup \lbrack x,d]$ is not a
geodesic.

On the other hand the nongeodesic $[b,x]\cup \lbrack x,d]$ is a subarc of
the geodesic $\alpha _{j}$ and we have a contradiction.

\textbf{Conclusion of case A2. }We have shown that $\alpha _{k}^{\symbol{94}%
}\in P_{k}$ and for all $j\leq k-1$ we have that $\alpha _{k}^{\symbol{94}}$ 
$\cap \alpha _{j}$ does not disconnect $\alpha _{k}^{\symbol{94}}.$ This
contradicts our original assumptions on $k$ and $i.$

\textbf{Case A3. }Suppose $c\notin G$ and $c\in H.$ Then repeat case 2 while
exchanging the roles of $a$ and $b.$

For the remaining cases, by symmetry, we lose no generality in assuming that 
$l([a,x])\leq l([b,x])$.

\textbf{B: The cases }$l[c,x]<l[a,x]$ or $l[d,x]<l[a,x].$

By symmetry we may assume $l[c,x]<l[a,x].$

Note $l[c,x]<l([b,x]).$ Recall $\alpha _{i}=[a,x]\cup \lbrack x,b].$

Let $G$ and $H$ be distinct components of $F_{i}$ such that $a\in G$ and $%
b\in H.$ Since $G\cap H=\emptyset ,$ $c\notin G\cap H.$

If $c\notin G$ then $[a,x]\cup \lbrack x,c]\in C_{i}$ and $l([a,x]\cup
\lbrack x,c])<l(\alpha _{i})$ contradicting our choice of $\alpha _{i}.$

If $c\in G$ then $c\notin H.$ Note $[c,x]\cup \lbrack x,b]\in C_{i}$ and $%
l([c,x]\cup \lbrack x,b])<l(\alpha _{i})$ and again we have a contradiction.

\textbf{C: The case }$l[c,x]>l[a,x]$ and $l[d,x]>l[a,x]$

Recall $\alpha _{k}=[c,x]\cup \lbrack x,d].$

Let $G$ and $H$ be distinct components of $F_{k}$ such that $c\in G$ and $%
d\in H.$ Since $G\cap H=\emptyset ,$ $a\notin G\cap H.$

If $a\notin G$ then $[a,x]\cup \lbrack x,c]\in C_{k}$ and $l([a,x]\cup
\lbrack x,c])<l(\alpha _{k})$ contradicting our choice of $\alpha _{k}.$

If $a\in G$ then $a\notin H.$ Note $[a,x]\cup \lbrack x,d]\in C_{k}$ and $%
l([a,x]\cup \lbrack x,d])<l(\alpha _{k})$ and again we have a contradiction.

\textbf{D: The cases }$(l[a,x]=l[c,x]\leq l[x,d])$ or $([a,x]=l[d,x]\leq
l[x,c]).$

By symmetry we assume $l[a,x]=l[c,x]\leq l[x,d].$ If $l[a,x]=l[b,x]$ then we
have treated this case already.

Thus we may assume $l[a,x]<l[b,x].$

It follows that $l[c,x]<l[d,x]$ since otherwise $l[b,x]=l[a,x]$ (since $%
l(\alpha _{i})\leq l(\alpha _{k})$).

Thus we may assume $l[c,x]<l([d,x])$. Let $G$ be the component of $F_{i}$
such that $a\in G.$

Suppose $c\notin G.$ Then $[a,x]\cup \lbrack x,c]\in C_{i}$ and $l([a,x]\cup
\lbrack x,c])<l(\alpha _{i})$ contradicting our choice of $\alpha _{i}.$

Suppose $c\in G.$ Then $[a,x]\cup \lbrack x,d]\in C_{k}$ and $l([a,x]\cup
\lbrack x,d])<l(\alpha _{k})$ contradicting our choice of $\alpha _{k}.$
\end{proof}

\begin{remark}
By construction $F_{k}$ has $n-k+1$ components, and since $\alpha _{n-1}$
exists $F_{n}=D_{1}\cup ..\cup D_{n}\cup \alpha _{1}..\cup \alpha _{n-1}$ is
connected.
\end{remark}

\subsection{\label{perturbarc}Perturbing the arcs $\{\protect\alpha _{i}\}$}

Appealing to section \ref{arcselect}, starting with a closed topological PL
disk $P\subset R^{2}$ and pairwise disjoint closed PL cellular sets $%
D_{1},..D_{n}$ (such that $D_{n}\subset P$) we have obtained a sequence of
closed arcs $\alpha _{1},,...\alpha _{n-1}$ such that $\partial \alpha
_{n}\subset \cup (\partial D_{i})$ and $int(\alpha _{i})\subset P\backslash
(\cup D_{i})$ and such that $D_{1}\cup ..\cup D_{n}\cup \alpha _{1}..\cup
\alpha _{n-1}$ is cellular (since if $A$ and $B$ are disjoint cellular
planar continua and $\beta \subset R^{2}$ is a closed arc such that $%
\partial \alpha \subset A\cup B$ and $int(\alpha )\subset R^{2}(A\cup B)$
and if $A\cup \alpha \cup \beta $ is connected then $A\cup \alpha \cup B$ is
cellular). It is also the case that $\alpha _{i}\cap \alpha _{j}$ is a
closed initial segment of $\alpha _{i}$ for all $i$ and $j.$

\subsubsection{Perturbing the interiors of \{$\protect\alpha _{i}$\}.}

Our first task is to perturb the interiors of the open arcs $int(\alpha
_{i}) $ to be disjoint, and we need the following Lemma.

\begin{lemma}
For each $i$ there exists $z_{i}\in int(\alpha _{i})$ such that $z_{i}\notin
\alpha _{j}$ for all $j\neq i.$
\end{lemma}

\begin{proof}
To obtain a contradiction suppose the Lemma is false. For some endpoint $%
a\in \alpha _{i},$ there exist $j\neq i$ such that $\alpha _{i}\cap \alpha
_{j}\neq \emptyset $ and such that $a\in \alpha _{i}\cap \alpha _{j}.$
Choose $j$ such that $\alpha _{i}\cap \alpha _{j}$ is a maximal initial
segment of $\alpha _{i}$ such that $a\in \alpha _{i}\cap \alpha _{j}.$

Let $\alpha _{i}=[a,c].$ Let $\alpha _{j}=[a,d]$ with $\alpha _{i}\cap
\alpha _{j}=[a,b]$ and note $b\in (a,c)$ since $[a,b]$ is a proper initial
segment of $[a,c].$ Obtain a sequence $b_{n}\rightarrow b$ such that $%
b_{n}\in (b,c].$

For each $b_{n}$ obtain $k_{n}\neq i$ such that $b_{n}\in \alpha _{k_{n}}.$
\ Since we have only finitely many closed arcs $\alpha _{j}$, there exists $%
k $ and $N$ such that $[b,b_{N}]\subset \alpha _{k}.$

Recall $\alpha _{k}\cap \alpha _{i}$ is a proper initial segment of $\alpha
_{i}$ and note $a\notin \alpha _{k}\cap \alpha _{i}$ since otherwise $%
l[a,b_{N}]>l([a,b])$ contradicting our choice of $j.$

Thus $[b,c]\subset \alpha _{k}$ and recall $k\neq i.$

Let $[x,c]=\alpha _{k}\cap \alpha _{i}$ with $x\in \lbrack a,b].$ Note $%
x\neq a$ since otherwise $\alpha _{j}\cap \alpha _{i}=\alpha _{i},$
contradicting the fact that $\alpha _{i}\cap \alpha _{k}$ is a proper subset
of $\alpha _{i}.$ Since $a\in \alpha _{j}\backslash \alpha _{k},$ we know $%
[d,b]\subset \alpha _{k}$ and hence

$\alpha _{j}\cap \alpha _{k}=[x,d].$ Thus $x=b$ since otherwise $\alpha _{k}$
is not a topological arc (since $\alpha _{k}=[x,b]\cup \lbrack b,d]\cup
\lbrack b,c]$). Thus $\alpha _{i}\cup \alpha _{j}\cup \alpha _{k}$ is a
triod, contradicting Lemma \ref{notriod}.
\end{proof}

For each $i,$ select a nonempty open arc $(w_{i},v_{i})\subset int(\alpha
_{i})$ such that $[w_{i},v_{i}]\cap \alpha _{j}=\emptyset $ if $i\neq j$ and
such that $\alpha _{i}\backslash (w_{i},z_{i})$ is the union of two arcs.

Let $\beta _{1},\beta _{2},...$ $\beta _{2(n-1)}$ denote the distinct arcs
determined by $\alpha _{i}\backslash (w_{i},z_{i}).$ Notice if $\beta
_{i}\cap \beta _{j}\neq \emptyset ,$ then $\beta _{i}$ and $\beta _{j}$
share a common endpoint on $\cup (D_{i})$ (since $\alpha _{i}\cap \alpha
_{j} $ is does not disconnect $\alpha _{i}$).

Moreover, since $[w_{i},v_{i}]\cap \alpha _{j}=\emptyset $ $\beta _{j}$ is
not a subset of $\beta _{i}.$ Consequently the arcs $\beta _{1},...\beta
_{2n}$ are canonically partitioned under the equivalence relation $\beta _{i}%
\symbol{126}\beta _{j}$ if $\partial \beta _{i}\cap \partial \beta _{j}\neq
\emptyset .$

Rename the arcs $\beta _{i}$ in format $\beta _{k}^{i}$ such that $%
\{e_{k}\}\in (\cup D_{i})\cap (\cap \beta _{k}^{i}).$ Thus $e_{k}$ is the
common endpoint in a given equivalence class.

By Lemma \ref{treelem}, for each $k$, the union of the arcs $\cup \beta
_{k}^{i}$ is a tree.

Let $h:P\rightarrow P$ be a tiny homeomorphism fixing $\cup D_{i}$ pointwise
and mapping $P\backslash (\cup D_{i})$ into $int(P).$

Notice $int(h(\beta _{k}^{i}))\subset int(P).$ Consequently, (since $h(\beta
_{i})\cap h(\beta _{i}))=\emptyset $ or $h(\beta _{i})\cap h(\beta _{i}))$
is a final segment of each arc) in similar fashion to the proof of Theorem 
\ref{shrtarcs}, it is apparent we can perturb, arbitrarily nearby, the
interiors $int(\beta _{k}^{i})$ while keeping fixed the endpoints $\partial
\beta _{k}^{i},$ creating new arcs $\beta _{k}^{\ast i}$ such that $%
int(\beta _{k}^{\ast i})=int(\beta _{l}^{\ast j})=\emptyset $ unless $k=l$
and $i=j.$ We can also require that $int(\beta _{k}^{\ast j})\cap
(v_{i},w_{i})=\emptyset $ for all $k,j,i.$

Now, to obtain perturbations of the interiors of the original arcs $\{\alpha
_{i}\},$ we assemble a perturbation of $\alpha _{i}$ as the union of the 3
arcs, $[w_{i},v_{i}]$ and the corresponding arcs $\beta _{k}^{\ast j}$ such
that $w_{i}\in \beta _{k}^{\ast j},$ and $\beta _{l}^{\ast t}$ such that $%
v_{i}\in \beta _{k}^{\ast j}$ and we let $\alpha _{i}^{\ast }=\beta
_{k}^{\ast j}\cup \lbrack w_{i},v_{i}]\cup \beta _{l}^{\ast t}.$

In this fashion we can perturb the arcs $\{\alpha _{i}\}$ and obtain arcs $%
\alpha _{1}^{\ast },,..\alpha _{n-1}^{\ast }$ such that $int(\alpha
_{i}^{\ast })\cap int(\alpha _{j}^{\ast })=\emptyset ,$ and $l(\alpha
_{i}^{\ast })<2\varepsilon $ and such that $\partial \alpha _{i}=\partial
\alpha _{i}^{\ast }.$

\subsubsection{Perturbing the endpoints of $\{\protect\alpha _{i}^{\ast }\}$}

Recall we have arcs $\alpha _{1}^{\ast },,..\alpha _{n-1}^{\ast }$ such that 
$int(\alpha _{i}^{\ast })\cap int(\alpha _{j}^{\ast })=\emptyset ,$ and $%
l(\alpha _{i}^{\ast })<2\varepsilon $ and such that $\partial \alpha
_{i}^{\ast }\subset (\cup D_{i}).$

To perturb the arcs $\{\alpha _{i}^{\ast }\}$ to be disjoint, let $%
\{e_{1},..,e_{m}\}=\cup \partial \alpha _{i}^{\ast }.$

Thus $e_{k}\in D_{i}\backslash int(D_{i}).$

For each $e_{k}$ we may select a small open set $U_{k}\subset P$ such that $%
e_{k}\in U_{k}$ and such that $(\cup (\alpha _{i}^{\ast }))\cap \overline{U}$
is a tree consisting of finitely many arcs $\beta _{k}^{\ast 1},...\beta
_{k}^{\ast n_{k}}$intersecting such that $\beta _{k}^{\ast i}\cap \beta
_{k}^{\ast j}=e_{k}.$

Let $\partial \beta _{k}^{^{\ast }i}=\{e_{k}^{i},b_{k}^{i}\}$ such that $%
e_{k}^{i}=e_{k}$ for all $i.$

It is topologically apparent the arcs $\{\beta _{k}^{\ast i}\}$ admit
arbitrarily small perturbations in $\overline{U}$ (fixing $b_{k}^{i}$) into
pairwise disjoint arcs $\beta _{k}^{\ast \ast i}$ such that $int(\beta
_{k}^{\ast \ast i})\subset int(P)$ and such that $e_{k}^{\ast \ast i}\in
\cup (D_{i}).$

As before, we assemble a perturbation $\alpha _{i}^{\ast \ast }$ of $\alpha
_{i}^{\ast }$ as the union of 3 arcs, a large closed subarc of $int(\alpha
_{i}^{\ast })$ and the two perturbations of the ends $\beta _{k}^{\ast \ast
i}\cup \beta _{l}^{\ast \ast j}.$

Now we have pairwise disjoint arcs $\alpha _{i}^{\ast \ast }$ such that $%
l(\alpha _{i})<2\varepsilon .$ Moreover $\cup (D_{i})\cup \alpha _{1}^{\ast
\ast }...\cup \alpha _{n-1}^{\ast \ast }$ is a nonseparating planar
continuum as seen in the following Lemma.

\begin{lemma}
Recall the arcs $\alpha _{1},,,\alpha _{n-1}$ selected by our algorithm in
the PL disk $P$ such that $\alpha _{i}^{\ast \ast }$ connects components of $%
\cup D_{i}$. Suppose we replace each arc $\alpha _{i}$ by an arc $\alpha
_{i}^{\ast \ast }$ such that $\alpha _{i}^{\ast \ast }\cap \alpha _{j}^{\ast
\ast }=\emptyset $ and such that $\alpha _{i}$ and $\alpha _{i}^{\ast \ast }$
connect the same components of $\cup D_{i}.$ Then $\cup D_{i}\cup \alpha
_{1}^{\ast \ast }..\cup \alpha _{n-1}^{\ast \ast }$ is a nonseparating
planar continuum.
\end{lemma}

\begin{proof}
This follows by induction. By definition each component of $\cup D_{i}$ is
cellular. Suppose each component of $\cup D_{i}\cup \alpha _{1}^{\ast \ast
}..\cup \alpha _{k-1}^{\ast \ast }$ is cellular. By construction the arc $%
\alpha _{k}^{\ast \ast }$ connects distinct components cellular components $%
D $ and $E$ of $\cup D_{i}\cup \alpha _{1}^{\ast \ast }...\cup \alpha
_{k-1}^{\ast \ast }$ such that $\alpha _{k}^{\ast \ast }\cap (D\cup E\cup
\alpha _{1}^{\ast \ast }..\cup \alpha _{k-1}^{\ast \ast })=\partial \alpha
_{k}^{\ast \ast }.$ It is topologically apparent that $D\cup \alpha
_{k}^{\ast \ast }\cup E$ is cellular.
\end{proof}

\subsubsection{Pushing $\{\protect\alpha _{i}^{\ast \ast }\}$ off of $\{%
\protect\gamma _{j}\}$}

We assume $D_{i}$ is the union of a PL disk $E_{i}$ and finitely many
pairwise disjoint arcs $\gamma _{i}^{1},...\gamma _{i}^{n_{i}}$ such that $%
E_{i}\cap \gamma _{i}^{j}=z_{ij}\in \partial \gamma _{i}^{j}.$

Suppose $\delta >0$ such that $l(\gamma _{i}^{j})<\delta .$ We assume that $%
E_{i}\subset int(P)$ and that $int(\gamma _{i}^{j})\subset int(P)$ and $%
\partial \gamma _{i}^{j}\backslash \{z_{ij}\}\subset \partial P.$

Recall we have pairwise disjoint arcs $\alpha _{i}^{\ast \ast }$ such that $%
l(\alpha _{i})<2\varepsilon $ and $\cup (D_{i})\cup \alpha _{1}^{\ast \ast
}...\cup \alpha _{n-1}^{\ast \ast }$ is a nonseparating planar continuum.

Unfortunately it is possible that $\partial \alpha _{k}^{\ast \ast }\cap
\gamma _{i}^{j}\neq \emptyset $ and we ultimately require that the
intersection is empty.

Thus the task at hand is to slide the arcs $\alpha _{k}^{\ast \ast }$ off of
the arcs $\{\gamma _{i}^{j}\}$.

Unlike the previous perturbations ($\alpha _{i}^{\ast }$ and $\alpha
_{i}^{\ast \ast }$ could be chosen arbitrarily close to $\alpha _{i}$) \
each end of $\alpha _{i}^{\ast \ast }$ might have to move by an amount $%
\delta .$

To see how to move the arcs $\alpha _{i}^{\ast \ast },$ first notice if $%
\alpha _{i}^{\ast \ast }\cap \gamma _{j}\neq \emptyset ,$ then $%
\{z_{i}\}=\alpha _{i}^{\ast \ast }\cap \gamma _{j}$ where $z_{i}\in \partial
\alpha _{i}^{\ast \ast }.$

Since each $D_{i}$ is the union of a disk $E_{i}$ with arcs $\gamma _{i}^{j}$
attached to $\partial E_{i},$ for each $\gamma _{i}^{j}$ we can select
pairwise disjoint sets $U_{i}^{j}$ (open in $P$) such that $\gamma
_{i}^{j}\subset U_{i}^{j}$, such that $\overline{U_{i}^{j}}\cap (\gamma
_{i}^{j}\cup E_{i})$ is a topological triod (and in particular $\overline{%
U_{i}^{j}}\cap E_{i}\subset \partial E_{i}$), and such that if $\alpha
_{k}^{\ast \ast }\cap U_{i}^{j}\neq \emptyset ,$ then $\alpha _{k}^{\ast
\ast }\cap U_{i}^{j}$ is connected.

We can also demand that $\overline{U_{i}^{j}}$ is a PL topological disk, and
that for each $x\in \overline{U}_{i}^{j}\backslash \gamma _{i}^{j}$ there
exists a path in $\overline{U_{i}^{j}}\backslash \gamma _{i}^{j}$ of length
less than $\delta $ connecting $x$ to $\partial E_{i}\backslash \gamma
_{i}^{j}.$

By construction , if $\alpha _{k}^{\ast \ast }\cap \overline{U_{i}^{j}}\neq
\emptyset $ then $\alpha _{k}^{\ast \ast }\cap \overline{U_{i}^{j}}$ is an
arc with precisely one end point $a_{kij}\in \gamma _{i}^{j}$ and the other
endpoint $b_{kij}\in \partial U_{i}^{j}\backslash (\cup D_{i}).$ Hence,
working entirely within $\overline{U_{i}^{j}},$ we can replace $\alpha
_{k}^{\ast \ast }\cap \overline{U_{i}^{j}}$ with an arc $\beta
_{ki}^{j}\subset \overline{U_{i}^{j}}$ such that $l(\beta _{ki}^{j})<\delta $
and such that $b_{kij}\in \partial \beta _{ki}^{j}$ and such that the other
endpoint of $\beta _{ki}^{j}$ is on $\partial E_{i}$ and such that $%
int(\beta _{ki}^{j})\cap (\cup D_{i})=\emptyset .$ It is apparent we can
preserve the disjointness property.

Ultimately this procedure replaces the arcs $\alpha _{k}^{\ast \ast }$ by
arcs $\alpha _{k}^{\ast \ast \ast }$, the union of 3 segments $\alpha
_{k}^{\ast \ast \ast }=(\alpha _{k}^{\ast \ast }\backslash (\beta
_{ki}^{j}\cup \beta _{kl}^{t}))\cup \beta _{ki}^{j}\cup \beta _{kl}^{t}.$

By construction $l(\alpha _{k}^{\ast \ast \ast })<2\varepsilon +2\delta ,$
and $\partial \alpha _{k}^{\ast \ast \ast }\subset \cup E_{i},$ and $%
int(\alpha _{k}^{\ast \ast \ast })\subset P\backslash (\cup D_{i}),$ and $%
\alpha _{i}^{\ast \ast \ast }\cap \alpha _{j}^{\ast \ast \ast }=\emptyset $
if $i\neq j,$ and $(\cup E_{i})\cup \alpha _{1}^{\ast \ast \ast }\cup ..\cup
\alpha _{k-1}^{\ast \ast \ast }$ is a cellular planar continuum (since $%
\partial \alpha _{i}^{\ast \ast }$ and $\partial \alpha _{i}^{\ast \ast \ast
}$ connect the same components of $\cup D_{i}$).

\section{Ingredients for proof of Theorem \ref{main4}}

\subsection{Standard planar Peano continua}

We clarify the structure of planar Peano continua and observe that any
planar Peano continuum $Y$ is homotopy equivalent to a `thicker' planar
Peano continuum $X$ so that the components of $R^{2}\backslash X$ have
simple closed curve boundaries. However, the closure of this null sequence
of circles (in some sense the `boundary' of $X$) need not have locally path
connected components. Similar observations are made in Theorem 2.4.1 \cite
{CC2}.

\begin{remark}
\label{peanorem}If $X\subset R^{2}$ then $X$ is a Peano continuum if and
only if the components $\{U_{n}\}$ of $(R^{2}\cup \{\infty \})\backslash X$
form a null sequence of simply connected open sets with locally path
connected boundary. (On the one hand if $X$ is a Peano continuum then each
of the components of $R^{2}\backslash X$ is open (since $X$ is compact) and
simply connected (since otherwise $X$ fails to be connected) and $\{U_{n}\}$
is a sequence (since $R^{2}$ is separable and open sets have at most
countably many components) and $diam(U_{n})\rightarrow 0$ (since otherwise
we can select a subsequence of large subcontinua $Z_{n}\subset U_{n},$
converging in the Hausdorff metric to a large subcontinuum $Z\subset X$ and
there exists $z\in Z$ such that local connectivity of $X$ fails) and $%
\partial U_{n}$ is locally path connected (since $R^{2}\backslash U_{n}$ is
locally path connected). Conversely if $X$ enjoys all of the above
properties then $X$ is compact (since $\cup U_{n}$ is open), and connected
(since each $U_{n}$ is open and simply connected), and locally path
connected (since for large $n$, there is a small retract from $\overline{%
U_{n}\backslash u_{n}}$ onto $\partial U_{n},$ and if $\{x,y\}\subset X,$
the short segment $[x,y]$ can be modified (replacing components of $%
[x,y]\cap U_{n}$ with the image in $\partial U_{n}$ under the retraction)
creating a small path in $X$ from $x$ to $y$).
\end{remark}

\bigskip Recall if $X$ $\subset R^{2}$ is a 2 dimensional Peano continuum
then $int(X)$ is the maximal set $U\subset X$ such that $U$ is open in $%
R^{2} $ and recall the frontier $Fr(X)=X\backslash int(X).$

It is tempting to conclude that each component of $Fr(X)$ is locally path
connected. However this is generally false as seen in the following example
(see also \cite{CC2} \cite{Z2}).

\begin{example}
\label{exbad}Let $Z$ be any 1 dimensional nonseparating planar continuum
such that $Z$ is not locally path connected. Manufacture a null sequence of
pairwise disjoint simple closed curves $C_{n}\subset R^{2}\backslash Z$ such
that $\overline{\cup C_{n}}=Z\cup (\cup C_{n})$ (and such that each $C_{n}$
is an isolated subspace of the compactum $Z\cup (\cup C_{n})).$ Let $U_{n}$
denote the bounded component of $R^{2}\backslash C_{n}$ and define $%
X=R^{2}\backslash (\cup U_{n}).$ Then $X$ is a Peano continuum (since the
components of $(R^{2}\cup \{\infty \})\backslash X$ form a null sequence of
simply connected open sets with locally connected boundary) however $Z$ is a
component of $Fr(X)$ and $Z$ is not locally path connected.
\end{example}

It will prove useful to obtain a canonical form (up to homotopy equivalence)
for planar Peano continua.

\begin{definition}
Suppose $X\subset R^{2}$ is a 2 dimensional Peano continuum. Then $X$ is 
\textbf{standard} if for each component $U\subset R^{2}\backslash X,$ $%
\partial U$ is a round Euclidean circle such that $\partial U$ is isolated
in $Fr(X).$
\end{definition}

\begin{lemma}
\label{stand}Suppose $Y\subset R^{2}$ is a Peano continuum. Then there
exists a standard Peano continuum $X$ such that $Y\subset X$ and $Y$ is a
strong deformation retract of $X.$ In particular $Y$ is homotopy equivalent
to $X.$
\end{lemma}

\begin{proof}
Since $R^{2}$ is separable, the open subspace $R^{2}\backslash X$ has at
most countably many components. Moreover each component $U\subset
R^{2}\backslash X$ is simply connected, since $X$ is connected. Let $%
\{U_{n}\}$ denote the simply connected components of $R^{2}\backslash Y.$
Note, for each simply connected $U_{n},$ $\partial U_{n}$ is locally path
connected. Select $u_{n}\in U_{n}$ and note $\partial U_{n}$ is a strong
deformation retract of $U_{n}\backslash \{u_{n}\}.$ In particular we can
select a simple closed curve $C_{n}\subset U_{n}$ approximating $\partial
U_{n}$ (let $C_{n}$ denote the image of a large round circle $S_{n}\subset
int(D^{2})$ under a Riemann map $\overline{\phi }:\overline{int(D^{2})}%
\rightarrow \overline{U_{n}}$ ). Let $A_{n}\subset \overline{U_{n}}$ denote
the closure of the open annulus bounded by $C_{n}$ and $\partial U_{n}.$
Note we have a strong deformation from $A_{n}$ onto $\partial U_{n}$ with
small trajectories under the homotopy. Let $Y=X\cup (\cup A_{n}).$ Since $X$
is a Peano continuum, $diam(U_{n})\rightarrow 0.$ The space $Y$ is a Peano
continuum by Remark \ref{peanorem}. Moreover, since $diam(A_{n})\rightarrow
0 $, the union of the deformation retracts from $A_{n}\rightarrow \partial
U_{n}$ determines that $X$ is a strong deformation retract of $Y.$ By
construction $Y$ is standard.
\end{proof}

\begin{lemma}
\label{ezkill}Suppose $X\subset R^{2}$ is a standard Peano continuum and $%
V\subset R^{2}\backslash X$ is the union of some bounded components of $%
R^{2}\backslash X.$ Then $X\cup V$ is a standard Peano continuum.
\end{lemma}

\begin{proof}
Let $Y=X\cup V.$ Then if $u\in R^{2}\backslash Y$ then there exists a
component $U\subset R^{2}\backslash X$ such that $u\in U$. Thus $Y$ is
compact and it follows from Remark \ref{peanorem} that $Y$ is a Peano
continuum whose complementary domain boundaries form a null sequence of
circles $S_{1},S_{2},...$

Recall $S_{n}$ is isolated in $X\backslash int(X).$ To see that $S_{n}$ is
isolated in $Y\backslash int(Y),$ note $V\subset int(Y)$ and thus $%
Y\backslash int(Y)\subset X\backslash int(X).$
\end{proof}

\subsection{Building and recognizing homotopic maps}

Notice if $\alpha $ and $\beta $ are two unbased inessential loops in $%
R^{2}\backslash \{(0,0)\}$ such that $diam(\alpha )<\varepsilon $ and $%
diam(\beta )<\varepsilon $ then $im(\alpha )\cup im(\beta )$ might have
large diameter. However if $\alpha $ and $\beta $ are homotopic essential
loops then we are guaranteed a small homotopy between $\alpha $ and $\beta $
since both loops must stay near the `hole' at $(0,0)$. (See also Theorem 2.1 
\cite{CC2} )

The foregoing example illustrates a more general phenomenon captured by the
following 2 Lemmas.

Our proof of Lemma \ref{shape} includes both a direct argument and a
backhanded proof exploiting the nontrivial fact that the fundamental group
of a planar Peano continuum $Z$ injects into the inverse limit of free
groups (determined by $Z$ as the nested intersection of a sequence open
planar sets).

\begin{lemma}
\label{shape} Suppose $Y\subset R^{2}$ is any set and $\alpha ,\beta
:S^{1}\rightarrow Y$ are essential homotopic loops such that $diam(im(\alpha
))<\varepsilon $ and $diam(im\beta )<\varepsilon .$ Then $diam(im\alpha \cup
im\beta )<2\varepsilon .$
\end{lemma}

\begin{proof}
Note $im(\alpha )$ is a Peano continuum. Let $Z_{\alpha }$ denote the union
of $im(\alpha )$ and those components $U\subset R^{2}\backslash im(\alpha )$
such that $U\cap Y=U.$ Then $Z_{\alpha }\subset Y$ and $Z_{\alpha }$ is a
Peano continuum ( by Lemma \ref{ezkill}). Since $Z_{\alpha }\subset Y$ it
follows that $\alpha $ is essential in $Z_{\alpha }.$ Note $diam(Z_{\alpha
})<\varepsilon .$ For each bounded component $V\subset R^{2}\backslash
Z_{\alpha }$ select a point $z_{V}\in V\backslash Y$ to obtain a set $%
E\subset R^{2}\backslash (Z_{\alpha }\cup Y).$ Note $Y\subset
R^{2}\backslash E.$ Thus $\alpha $ and $\beta $ are homotopic in $%
R^{2}\backslash E.$

To see that $\alpha $ is essential in $R^{2}\backslash E$ it suffices to
prove that $Z_{\alpha }$ is a strong deformation retract of $R^{2}\backslash
E.$ (For each bounded component $V\subset R^{2}\backslash Z_{\alpha },$
notice $\overline{V\backslash z_{v}}$ can be deformation retracted onto $%
\partial V,$ and since the components of $R^{2}\backslash Z_{\alpha }$
determine a null sequence it follows, taking the union of the $SDRs,$ that $%
Z_{\alpha }$ is a strong deformation retract of $R^{2}\backslash E$).

(Alternately we can obtain a finite set $E_{n}\subset E$ such that $\alpha $
is essential in $R^{2}\backslash E_{n}$ as follows. Obtain PL nested compact
polyhedra ..$P_{3}\subset P_{2}\subset P_{1}$ such that $Z_{\alpha }\subset
\cap _{n=1}^{\infty }P_{n}.$ Inclusion determines a canonical homomorphism $%
\phi :\pi _{1}(Z_{\alpha })\rightarrow \lim_{\leftarrow }\pi _{1}(P_{n}),$
it is a nontrivial fact that $\phi $ is injective (established more
generally for planar sets \cite{FZ}).Thus there exists $N$ such that $\alpha 
$ is essential in $P_{N}$ and such that $diam(P_{N})<\varepsilon ,$ and now
select a point from each bounded component of $R^{2}\backslash P_{N}$ to
obtain $E_{N}$).

Recall $\alpha $ and $\beta $ are essential homotopic loops in $%
R^{2}\backslash E$ and $diam(E)<\varepsilon .$ Let $B\subset R^{2}$ denote
the convex hull of the set $E.$ Note $diam(B)=diam(E).$ It is apparent that $%
B\cap im(\beta )\neq \emptyset $, (since otherwise $\beta $ would be
inessential in $R^{2}\backslash E).$ Thus $diam(im\alpha \cup im\beta
)<2\varepsilon .$
\end{proof}

\begin{lemma}
\label{retract}Suppose $Y$ is a planar set, $\varepsilon >0,$ and $A$
denotes the closed annulus $S^{1}\times \lbrack 0,1].$ Suppose $%
h:A\rightarrow Y$ is a map such that $diam(h(\partial A))<\varepsilon .$
Then there exists a map $H:A\rightarrow Y$ such that $diam(H(A))<\varepsilon 
$ and $h_{\partial A}=H_{\partial A}.$ Suppose $\alpha :S^{1}\rightarrow Y$
is inessential and suppose $diam(im(\alpha ))<\varepsilon .$ Then there
exists a map $\beta :D^{2}\rightarrow Y$ such that $\beta _{S^{1}}=\alpha $
and $diam(im(\beta ))<\varepsilon .$
\end{lemma}

\begin{proof}
Let $U$ be the unbounded component of $R^{2}\backslash h(\partial A).$ Note $%
diam(P)<\varepsilon $ and $P$ is a simply connected Peano continuum and
hence there exists a retract $R:R^{2}\rightarrow P$ such that $R_{P}=id_{P}.$
Let $H=R(h).$

Let $V$ be the unbounded component of $R^{2}\backslash im(\alpha ).$ Let $%
Q=R^{2}\backslash V.$ Then $Q$ is a simply connected Peano continuum and
hence there exists a retract $r:R^{2}\rightarrow Q.$ Let $\gamma
:D^{2}\rightarrow Y$ be any map extending $\alpha $ and let $\beta =r(\gamma
).$
\end{proof}

The following elementary Lemma is essentially the Alexander Trick and
ensures we can canonically adjust a map of a disk while keeping half the
boundary unadjusted.

\begin{lemma}
\label{blow}Suppose $Y$ is any metric space and $D$ is a topological disk
and $p\in \partial D$ and $\gamma \subset \partial D$ is closed arc such
that $p\notin \gamma .$ Suppose $f:D\rightarrow Y$ is a map. Suppose $\alpha
:[0,1]\rightarrow Y$ is a path connecting $f(p)$ and $q\in Y.$ Then there
exists a map $g:D\rightarrow Y$ such that $g(p)=q$ and $g_{\gamma
}=f_{\gamma }$ and there exists a homotopy from $f$ to $g$ such that the
trajectories in $Y$ under the homotopy have diameter bounded by $%
diam(im(\alpha )\cup im(f)).$ Moreover $diam(g(D))\leq diam(f(D)\cup
im(\alpha )).$
\end{lemma}

\begin{proof}
We may assume $D\subset R^{2}$ is the closed upper half disk of radius $1$
centered at $(0,0)=p,$ and $\gamma \subset \partial D$ is the semicircle of
radius $1.$

Let $E\subset R^{2}$ denote the closed unit disk. Define $F:E\rightarrow Y$
so that $F(x,-y)=f(x,y).$ Notice for each $z_{\theta }\in \partial E$ there
is a canonical path in $Y$ connecting $f(z_{\theta })$ and $q$, (we let $%
\beta _{\theta }$ denote the radial segment connecting $z_{\theta }$ and $%
(0,0)$ and we let $\gamma _{\theta }=f(\beta _{\theta })\ast \alpha .)$

Define $G:E\rightarrow Y$ so that $G$ maps $\beta _{\theta }\subset E$
linearly onto $\gamma _{\theta }\subset Y.$ Let $\alpha _{s}$ denote a
homotopy in $Y$ from $p$ to $\alpha $ so that $im(\alpha _{s})\subset
im(\alpha ).$

For $s\in \lbrack 0,1]$ define $\beta _{\theta }^{1}(s)$ and $\beta _{\theta
}^{2}(s)$ so that $\beta _{\theta }=\beta _{\theta }^{1}(s)\ast \beta
_{\theta }^{2}(s)$ (concatenated segments varying linearly with $s$) so that 
$\beta _{\theta }^{1}(0)=\beta _{\theta }$ and $\beta _{\theta
}^{2}(0)=(0,0).$ To obtain a homotopy from $F$ to $G$ let $F_{s}$ map $\beta
_{\theta }^{1}(s)\ast \beta _{\theta }^{2}(s)$ onto $f(\beta _{\theta })\ast
\alpha _{s}$ `homomorphically'. Let $g=G_{D}$ and let $f_{s}=F_{sD}.$
\end{proof}

\begin{theorem}
\label{pointadjust}Suppose $f:X\rightarrow Y$ is a map of a standard Peano
continuum $X$ and suppose $(Y,d)$ is any metric space. Suppose $\{x_{n}\}$
is a sequence of distinct points in $X$ (and each $x_{n}$ belongs to some
isolated boundary circle $C_{m}\subset Fr(X)$ and each boundary circle
contains at most finitely many of $\{x_{n}\}$) and suppose $y_{n}$ is a
sequence in $im(f)\subset Y$ and suppose $d(f(x_{n}),y_{n})\rightarrow 0.$
Then there exists a map $f^{\symbol{94}}:X\rightarrow Y$ such that $f^{%
\symbol{94}}$ is homotopic to $f$ and $f^{\symbol{94}}(x_{n})=y_{n}.$
\end{theorem}

\begin{proof}
Since $im(f)$ is a Peano continuum, there exists a null sequence of paths $%
\alpha _{n}:[0,1]\rightarrow im(f)$ connecting $f(x_{n})$ and $y_{n}.$
Select a null sequence of pairwise disjoint closed topological disks $%
\{D_{n}\}$ such that $x_{n}\in D_{n}\subset X$ and such that $(\partial
D_{n})\cap Fr(X)$ is a nontrivial arc containing $x_{n}$ in its interior.
Let $\gamma _{n}=\partial D_{n}\backslash (int((\partial D_{n})\cap Fr(X))).$
Let $f^{\symbol{94}}$ and $f$ agree over the set $X\backslash (\cup
int(D_{n})).$ Apply Lemma \ref{blow} to the data $(D_{n},\gamma
_{n},x_{n},\alpha _{n})$ , and sew together the resulting maps to obtain $f^{%
\symbol{94}}.$
\end{proof}

\begin{lemma}
\label{frfix}Suppose $f,g:X\rightarrow Y$ are maps of a standard planar
Peano continuum $X$ into the metric space $Y.$ Suppose for each isolated
circle $C_{n}\subset Fr(X)$ there exists $x_{n}\in C_{n}$ such that $%
f(x_{n})=g(x_{n}).$ Suppose $f_{\ast }=g_{\ast }$ and $f_{\ast },g_{\ast
}:\pi _{1}(X,x_{1})\rightarrow \pi _{1}(Y,f(x_{1}))$ are the induced
homomorphisms. Suppose $f(C_{n})$ is essential in $Y$ for each isolated
circle $C_{n}\subset Fr(X).$ Then there exists a map $f^{\symbol{94}%
}:X\rightarrow Y$ such that $f$ is homotopic to $f$ and $f_{Fr(X)}^{\symbol{%
94}}=g_{Fr(X)}.$
\end{lemma}

\begin{proof}
Since $f_{\ast }=g_{\ast },$ for each $C_{n}$ the loops $f_{C_{n}}$ and $%
g_{C_{n}}$ are essential in $Y$ and path homotopic in $Y.$ It follows from
Lemma \ref{shape} that $\lim diam_{n\rightarrow \infty }((f(C_{n}))\cup
g(C_{n}))\rightarrow 0.$

It follows from Lemma \ref{retract} that there exists a path homotopy
connecting $f(C_{n})$ and $g(C_{n})$ so that the image of the path homotopy
has diameter bounded by $diam((f(C_{n}))\cup g(C_{n}))+\frac{1}{2^{n}}.$

Since $X$ is a standard Peano continuum, we may select pairwise disjoint
closed round disks $D_{n}\subset R^{2}$ such that $C_{n}\subset D_{n},$ such
that $C_{n}\cap \partial D_{n}=\{x_{n}\}$ and such that $diam(D_{n})%
\rightarrow 0.$

Note the based loops $f_{\partial D_{n}},f_{C_{n}}$ and $g_{C_{n}}$ are
homotopic (fixing $x_{n}$ throughout), and By Lemma \ref{retract} for large $%
n$ the homotopies can be chosen to be small.

In particular, in the pinched annulus $A_{n}$ bounded by $C_{n}\cup \partial
D_{n}$ we can define $f^{\symbol{94}}$:$A_{n}\rightarrow Y$ so that $%
f_{\partial D_{n}}^{\symbol{94}}=f_{\partial D_{n}},$ and so that $%
f_{C_{n}}^{\symbol{94}}=g_{C_{n}}$ and so that $diam(f^{\symbol{94}%
}(A_{n}))\rightarrow 0.$ Let $\ f$ and $f^{\symbol{94}}$ agree on the set $%
X\backslash (\cup int(A_{n}))$ and we have obtained the desired map $f^{%
\symbol{94}}.$
\end{proof}

\begin{lemma}
\label{circlekill}Suppose $X\subset R^{2}$ is a standard Peano continuum,
suppose $Y$ is any metric space and $f:X\rightarrow Y$ a map$.$ Suppose $%
U\subset R^{2}\backslash X$ is the union of those bounded components $%
\{U_{n}\}\subset R^{2}\backslash X$ such that $f_{\partial U_{n}}$ is
inessential in $Y.$ Let $Z=X\cup U.$ Then $Z$ is a standard Peano continuum
and there exists a map $F:Z\rightarrow Y$ such that $F_{X}=f$ and $F_{C_{n}}$
is essential for all map isolated circles $C_{n}\subset Fr(Z).$
\end{lemma}

\begin{proof}
By Lemma \ref{ezkill} that $Z$ is a standard Peano continuum.

For each component $U_{n}\subset U$, the unbased loop $f_{\partial U_{n}}$
inessential in $Y$, and thus there exists an extension $F_{\overline{U_{n}}%
}\rightarrow Y$ such that $F_{\partial U_{n}}=f_{\partial U_{n}}.$ By Lemma 
\ref{retract} we can also require that $diam(F(U_{n}))\leq diam(f(\partial
U_{n})).$ Let $F$ and $f$ agree over $X,$ since $diam(U_{n})\rightarrow 0$
and since $f$ is uniformly continuous on $X,$ it follows that $%
diam(f(U_{n}))\rightarrow 0$ and hence the extension $F$ is continuous. By
construction $f_{C_{n}}=F_{C_{n}}$ for all isolated circles $C_{n}\subset
\partial Fr$ and hence $F_{C_{n}}$ is essential in $Y.$
\end{proof}

\begin{lemma}
\label{fg}Suppose $X\subset R^{2}$ is a standard Peano continuum, suppose $Y$
is any metric space and $p\in X$ and $f,g:X\rightarrow Y$ are maps such that 
$f(p)=g(p)$ and $f_{\ast }=g_{\ast }$ and $f_{\ast }:\pi
_{1}(X,p)\rightarrow \pi _{1}(Y,f(p))$ is the induced homomorphism$.$
Suppose $U\subset R^{2}\backslash X$ is the union of those bounded
components $\{U_{n}\}\subset R^{2}\backslash X$ such that $f_{\partial
U_{n}} $ is inessential in $Y.$ Suppose $Z=X\cup U$ and suppose the maps $%
F,G:Z\rightarrow Y$ satisfy $F_{X}=f$ and $G_{X}=g.$ \ Then $G_{\ast
}=F_{\ast }$ (and $F_{\ast }:\pi _{1}(Z,p)\rightarrow \pi _{1}(Y,f(p))$
denotes the induced homomorphism).
\end{lemma}

\begin{proof}
Suppose $\alpha :[0,1]\rightarrow Z$ is a loop based at $p.$ By Lemma \ref
{ezkill} $Z$ is a Peano continuum. For each component $U_{n}\subset U$ and
each component $J\subset \alpha ^{-1}(U_{n})$ replace $\alpha _{J}$ by $%
\beta _{J}$ such that $\beta _{\overline{J}}$ is path homotopic to $\alpha _{%
\overline{J}}$ in $Z$ and such that $im(\beta _{J})\subset \partial U_{n}$
and such that $diam(im(\alpha _{j}))=diam(im(\beta _{j})).$ Let $\beta =\cup
_{J}\alpha _{([0,1]\backslash \cup J)}\cup \beta _{J}).$

Since the collection of all such open arcs $J$ is a null sequence of
intervals in $[0,1],$ and since $\alpha $ is uniformly continuous, the
homotopies connecting $\alpha _{J}$ to $\beta _{J}$ can be chosen to be
small, and the union of the homotopies determines that $\alpha $ and $\beta $
are path homotopic in $Z.$ Thus $F(\alpha )$ and $F(\beta )$ are path
homotopic in $Y$ and $G(\alpha )$ and $G(\beta )$ are path homotopic in $Y.$

Note $f(\beta )=F(\beta )$ and $g(\beta )=G(\beta )$ and, (since $f_{\ast
}=g_{\ast }$), $f(\beta )$ and $g(\beta )$ are path homotopic in $Y.$ Thus $%
F(\alpha )$ and $G(\alpha )$ are path homotopic in $Y.$
\end{proof}

\begin{lemma}
\label{cherry}Suppose $X$ is a planar continuum and $Y\subset R^{2}$ is any
planar set and $Z\subset X$ is a continuum and $f,g:X\rightarrow Y$ are maps
such that $f_{Z}=g_{Z}.$ Then $f$ and $g$ are homotopic if both the
following conditions hold 1) Each bounded component $U\subset X\backslash Z$
is an open Jordan disk. (i.e. $\overline{U}\backslash U$ is a simple closed
curve) and 2) If $X\backslash Z$ has infinitely many components $%
U_{1},U_{2},..$ then $diam(U_{n})\rightarrow 0.$
\end{lemma}

\begin{proof}
For each bounded component $U\subset X\backslash Z$ select embeddings $h_{u}:%
\overline{U}\hookrightarrow \partial B^{3}$ and $h_{l}:\overline{U}%
\rightarrow \partial B^{3}$ mapping $\overline{U}$ respectively onto the
upper and lower hemispheres of the $2-$sphere. Glue the maps together to
obtain a map of the 2-sphere $j_{U}=$ $f(h_{u}^{-1})\cup
g(h_{l}^{-1}):\partial B^{3}\rightarrow Y.$ Since planar set are aspherical 
\cite{CCZ}, $j_{U}$ is the restriction of a map $J_{U}:B^{3}\rightarrow
im(j_{U}).$ The map $J_{U}$ determines a homotopy $\phi _{U}^{t}$ between $%
f_{\overline{U}}$ and $g_{\overline{U}}$ and the image of the homotopy lies
in $im(f_{\overline{U}})\cup im(g_{\overline{U}}).$ If $X\backslash Z$ has
finitely many components then we take the union of the homotopies. Since
each of $f$ and $g$ are uniformly continuous (since $X$ is compact) , and
since $diam(U_{n})\rightarrow 0,$ it follows that the union of the
homotopies $f_{z}\cup (\cup _{n}\phi _{U_{n}}^{t})$ determines a global
homotopy between $f$ and $g.$
\end{proof}

\begin{lemma}
\label{arcfix}Suppose $X\subset R^{2}$ is a 2 dimensional continuum and $%
\beta _{1},\beta _{2},...$is null sequence of closed arcs such that $%
int(\beta _{n})\subset int(X)$ , and $\partial \beta _{n}\subset Fr(X)$ and
such that $int(\beta _{n})\cap int(\beta _{m})=\emptyset $ if $n\neq m.$
Suppose $Y$ is a planar set. Suppose $f,g:X\rightarrow Y$ are maps such that 
$f_{Fr(X)}=g_{Fr(X)}.$ Suppose $p\notin \cup \beta _{n}$ and $f(p)=g(p).$
Suppose $f_{\ast }=g_{\ast }$ and $f_{\ast }:\pi _{1}(X,p)\rightarrow \pi
_{1}(Y,f(p))$ is the induced homomorphism. Let $Z=Fr(X)\cup (\cup \beta
_{n}).$ Then there exists a map $f^{\symbol{94}}:X\rightarrow Y$ such that $%
f^{\symbol{94}}$ is homotopic to $f$ and $f_{Z}^{\symbol{94}}=g_{Z}.$
\end{lemma}

\begin{proof}
`Thicken' each $\beta _{n}$ into a closed topological disk $D_{n}$ (such
that $p\notin D_{n})$ and $int(\beta _{n})\subset int(D_{n})\subset int(X),$
and $D_{n}\cap Fr(X)=\partial \beta _{n},$ and if $n\neq m$ then $D_{n}\cap
D_{m}\subset \partial \beta _{n}\cup \partial \beta _{m},$ and $%
diam(D_{n})\rightarrow 0.$ Since $f_{\ast }=g_{\ast }$ it follows that $%
f_{\beta _{n}}$ and $g_{\beta _{n}}$ are path homotopic in $Y$ (fixing $%
\partial \beta _{n}$ throughout the homotopy).

Note $\beta _{n}$ is a strong deformation retract of $D_{n}$ and if (the
topological semicircle) $\gamma $ is the closure of either component of $%
(\partial D_{n})\backslash (\partial \beta _{n})$ then $f_{\gamma }$ is path
homotopic to $f_{\beta _{n}}.$ Consequently we can define $f_{D_{n}}^{%
\symbol{94}}:D_{n}\rightarrow Y$ so that $f_{\partial D_{n}}^{\symbol{94}%
}=f_{\partial D_{n}}$ and so that $f_{\beta _{n}}^{\symbol{94}}=$ $g_{\beta
_{n}}$. By Lemma \ref{retract} we can also arrange $diam(f^{\symbol{94}%
}(D_{n}))<diam(f(\partial D_{n})\cup g(\beta _{n})).$

Now sew together two copies of $D_{n}$ joined along $\partial D_{n}$ and
notice$\ f_{D_{n}}^{\symbol{94}}\cup f_{D_{n}}:S^{2}\rightarrow Y$
determines a map of the 2 sphere into $Y.$ Moreover $Z_{n}=f^{\symbol{94}%
}(D_{n})\cup f(D_{n})$ is a Peano continuum (since $S^{2}$ is a Peano
continuum). Since $Z_{n}$ is aspherical (\cite{CCZ}) $f_{D_{n}}^{\symbol{94}%
} $ and $f_{D_{n}}$ are homotopic in $Z_{n}$ (via a homotopy $f_{n}^{t}$
such that $f_{n|\partial D_{n}}^{t}=f_{\partial D_{n}}$).

Define $f_{X\backslash (\cup int(D_{n}))}^{\symbol{94}}=f_{X\backslash (\cup
int(D_{n}))}.$ For large $n$ the homotopy $f_{n}^{t}$ has small image (since 
$f$ and $g$ are uniformly continuous and both $D_{n}$ and $Z_{n}$ are null
sequences).

Thus the union of the homotopies $f_{n}^{t}$ shows $f^{\symbol{94}}$ and $f$
are homotopic.
\end{proof}

\subsection{Chopping up simply connected sets into small disks with crosscuts%
}

Given a PL planar disk $D$ with finitely marked points $Y\subset \partial D$
we wish to partition $D$ using crosscuts (with disjoint interiors)
connecting distinct points of $Y,$ and to obtain control of the diameter of
the regions bounded by the crosscuts (Theorem \ref{markedchop}). Combined
with a standard construction from the theory of prime ends (Lemma \ref
{cancat}), we see how to subdivide simply connected open sets $U\subset S^{2}
$ into a null sequence of topological disks whose boundaries contain points
of $\partial U$ (Theorem \ref{main2}).

Define the planar set $Y\subset R^{2}$ as $`\varepsilon $ thin' if (letting $%
B(x,\varepsilon )\subset R^{2}$ denote the round open disk of radius $%
\varepsilon $) for each $x\in Y,$ $B(x,\varepsilon )\cap (R^{2}\backslash
Y)\neq \emptyset .$

If $D$ is a topological disk an arc $\alpha \subset D$ is a \textbf{spanning
arc }if $int(\alpha )\subset int(D)$ and $\partial \alpha \subset \partial
D. $ Let $S(v,\varepsilon )$ denote the round circle of radius $\varepsilon $
centered at $v.$

\begin{lemma}
\label{ethinlem}Suppose $...D_{3}\subset D_{2}\subset D_{1}\subset R^{2}$ is
a sequence of closed topological disks. Suppose $\varepsilon >0.$ Then there
exists $N$ so that if $n\geq N$ and $U$ is a component of $D_{n}\backslash
D_{n+1}$ then $U$ is $\varepsilon $ thin.
\end{lemma}

\begin{proof}
Suppose otherwise to obtain a contradiction. There exists an increasing
sequence $n_{k}\rightarrow \infty $ and $u_{n_{k}}\in U_{n_{k}}$ so that $%
U_{n_{k}}$ is a component of $D_{n_{k}}\backslash D_{n_{k}+1}$ and $%
B(u_{n_{k}},\varepsilon /2)\subset U_{n_{k}}.$ Let $z$ be a subsequential
limit of $\{u_{n_{k}}\}.$ Then $B(z,\frac{\varepsilon }{4})\subset U_{n_{k}}$
for all sufficiently large $k.$ Thus $z\notin \cap _{n=1}^{\infty }D_{n}$
and $z\in \cap _{n=1}^{\infty }D_{n}$ and we have a contradiction.
\end{proof}

If $S\subset R^{2}$ is a simple closed curve and $P\subset S$ is finite then 
$P$ is $`\varepsilon -$ dense' in $S$ if for each pair of consecutive
clockwise points $x<y\in P$ each clockwise arc $\alpha _{xy}\subset S$
connecting $x$ to $y$ satisfies $diam(\alpha _{xy})<\varepsilon .$

The inequalities in Lemma \ref{span} and Theorem \ref{markedchop} are not
sharp.

\begin{lemma}
\label{span}If $D\subset R^{2}$ is any $\varepsilon $ thin topological disk,
then there exist finitely many pairwise disjoint spanning arcs $\alpha
_{1},\alpha _{2},..$ such that if $U$ is a component of $D^{2}\backslash
(\cup \alpha _{n}),$ then $diam(U)<12\varepsilon .$
\end{lemma}

\begin{proof}
Tile the plane by squares of sidelength $5\varepsilon .$ Note there exist
arbitrarily small perturbations $E$ of $D$ so that $E$ is a PL disk, and so
that if $v\in T$ is a corner of the tile $T$ then $v\notin \partial E,$ and
so that each component $A\subset E\cap \overline{B(v,2\varepsilon )}$ is a
closed topological disk. Thus, wolog we may assume $D$ also enjoys the
aforementioned properties of $E$. Suppose $v$ is the corner of a tile $T$
and suppose $v\in int(D).$ Then, since $D$ is $\varepsilon $ thin, let the
open arc $\gamma $ be a nonempty component of $S(v,\varepsilon )\backslash D$
and let $B$ denote the component of $\overline{B(v,\varepsilon )}\backslash
D $ such that $\gamma \subset B.$ Now manufacture a homeomorphism $h:%
\overline{B(v,\varepsilon )}\rightarrow \overline{B(v,\varepsilon )}$
(fixing $\partial \overline{B(v,\varepsilon )}$ pointwise) so that $v\in
h(B).$ Applying this construction at the corner of each tile, ultimately we
can obtain a $2\varepsilon $ homeomorphism $h:R^{2}\rightarrow R^{2}$ so
that $h(D)$ is a PL disk and so that $v\notin h(D)$ for all tile corners $v,$
and so that $\partial h(D)$ crosses each tile edge finitely many times and
transversely.

Let $E=h(D).$ Recall the tiles $T_{1},T_{2},...$and note each component $%
A\subset E\backslash (\cup _{n=1}^{\infty }\partial T_{n})$ satisfies $%
diam(A)<5\sqrt{2}\varepsilon .$ We obtain spanning arcs $\beta _{1},\beta
_{2},,,$ taking the components of $E\cap (\cup _{n=1}^{\infty }(\partial
T_{n})).$ Now let $\alpha _{n}=h^{-1}(\beta _{n}).$ Let $U$ be a component
of $D\backslash (\cup \beta _{n})$ and note $diam(h^{-1}U)<5\sqrt{2}%
\varepsilon +4\varepsilon <12\varepsilon .$
\end{proof}

\begin{lemma}
\label{clock}Suppose the finite set $P$ is $\varepsilon $ dense in $\partial
D$ and $Q\subset \partial D$ is finite. Then there exists a monotone map:$%
f:D\rightarrow D$ such that $f(Q)\subset P$ and $d(f(x),x)<\varepsilon $ for
all $x\in D$ and such that $f$ maps $int(D)$ homeomorphically onto $int(D).$
\end{lemma}

\begin{proof}
For each $q\in Q$ let $f(q)=p\in P$ so that $p$ is the nearest clockwise
neighbor to $q$. For each $p\in P$ select an arc $\gamma _{p}\subset
\partial D$ so that $p\cup f^{-1}(p)\subset int(\gamma _{p})$ and so that $%
diam(\gamma _{p})<\varepsilon .$ We can also arrange that $\gamma _{p}\cap
\gamma _{r}=\emptyset $ if $p\neq r.$ Thicken the arcs $\{\gamma _{p}\}$
into a closed pairwise disjoint topological disks $\{D_{p}\}$ such that $%
\cup _{p\in P}D_{p}\subset D.$ Let $f$ fix $D\backslash (\cup D_{p})$
pointwise. For each $\gamma _{p}$ let $\beta _{p}\subset int(\gamma _{p})$
be closed arc such that $p\cup f^{-1}(p)\subset int(\gamma _{p}).$ Let $f$
fix $\partial D_{i}$ pointwise, let $f$ map $\beta _{p}$ to $p$ and let $f$
map $D_{p}\backslash \beta _{p}$ homeomorphically onto $D_{p}\backslash p.$
\end{proof}

\begin{theorem}
\label{markedchop}There exists $M>0$ so that if $D\subset R^{2}$ is an $%
\varepsilon $ thin topological disk and the finite set $P\subset \partial D$
is $\varepsilon $ dense then there exist spanning arcs $\beta _{1},\beta
_{2},...$ such that $int(\beta _{n})\cap int(\beta _{m})=\emptyset $ and
such that $P=\cup (\partial \beta _{n})$ and such that each component $%
U\subset D^{2}\backslash (\cup \beta _{n})$ is an open Jordan disk and $%
diam(U)<14\varepsilon .$
\end{theorem}

\begin{proof}
Obtain disjoint closed spanning arcs $\alpha _{1},\alpha _{2},$ as in Lemma 
\ref{span} so that each component $U\subset D^{2}\backslash (\cup \alpha
_{n})$ has diameter at most $12\varepsilon .$

Apply Lemma \ref{clock} to obtain a monotone map:$f:D\rightarrow D$ such
that $f(\cup \alpha _{n})\subset P$ and $d(f(x),x)<\varepsilon $ for all $%
x\in D$ and such that $f$ maps $int(D)$ homeomorphically onto $int(D).$

Let $\beta _{n}=f(\alpha _{n}).$ If $U$ is a component of $D\backslash (\cup
\alpha _{n})$ then $diam(f(U))<12\varepsilon +2\varepsilon .$ At this stage
it is likely that $P\backslash (\cup \partial \beta _{n})\neq \emptyset .$
In this case we merely add more spanning arcs to the collection $\beta
_{1},\beta _{2},..$ connecting any remaining points of $P.$
\end{proof}

\begin{definition}
Suppose $U\subset S^{2}$ is open and simply connected. Recall a \textbf{%
closed crosscut }$\gamma $ is a closed nontrivial topological arc such that $%
\gamma \subset \overline{U}$ and $int(\gamma )\subset U$ and $\partial
\gamma \subset \partial U.$ By a loop of \textbf{concatenated crosscuts }$%
\gamma _{1},\gamma _{2},..\gamma _{n}$ we mean each $\gamma _{i}$ is a
closed crosscut of $U$ and $int(\gamma _{i})\cap int(\gamma _{j})=\emptyset $
if $i\neq j$ and $\cup \gamma _{i}$ is a simple closed curve.
\end{definition}

Suppose $V\subset R^{2}$ is open, bounded, and simply connected and $\delta
>0.$ Let $T_{1}^{\delta }.,T_{2}^{\delta },...$ be a tiling of $R^{2}$ by
squares of sidelength $\delta .$ Let $A_{\delta }\subset V$ be a maximal
closed topological disk consisting of the union of finitely many closed
tiles. For small $\delta $ it is apparent that each point of $\partial A$
can be connected to $\partial V$ within $V$ by a crosscut of diameter at
most $2\delta $ (since otherwise we could attach another tile to $A$).
Consequently we have the following Lemma which will be obvious to the reader
familiar with prime ends.

\begin{lemma}
\label{cancat}Suppose $X\subset R^{2}$ is a nonseparating continuum and $%
\alpha $ is a crosscut of $R^{2}\backslash X$ and $V$ is the bounded
component of $R^{2}\backslash (X\cup \alpha ).$ Suppose $\varepsilon >0.$
There exist finitely many crosscuts $\beta _{1},\beta _{2},...$ of $V$ such
that $\alpha \cup (\cup \beta _{i})$ is a simple closed curve and $%
diam(\beta _{i})<\varepsilon .$
\end{lemma}

Given a cellular continuum $X\subset R^{2}$ and applying Lemma \ref{cancat}
recursively (applied to $\varepsilon _{n}=\frac{1}{2^{n}}$) we can
manufacture a sequence of closed topological disks $...D_{3}\subset
D_{2}\subset D_{1}\subset R^{2}$ such that 1) $X=\cap _{n=1}^{\infty }D_{n}$
and 2) each component $V\subset D_{n}\backslash D_{n+1}$ is an open planar
set such that $\partial V$ is a loop of finitely many crosscuts (of $%
R^{2}\backslash X$) $\gamma _{1},\gamma _{2},....$ such that $\gamma
_{1}\subset \partial D_{n},$ and $diam(\gamma _{1})<\varepsilon _{n},$ and $%
\gamma _{2}\cup \gamma _{3}..\subset \partial D_{n+1},$ and $diam(\gamma
_{i})<\varepsilon _{n+1}$ if $i\geq 2.$

By Lemma \ref{ethinlem}, given the disks $D_{1},D_{2},...$ and $\varepsilon
>0$ notice there exists $N$ so that if $n\geq N$ then $D_{n}\backslash
D_{n+1}$ has $\varepsilon $ thin (components).

Moreover by construction for each component $V\subset D_{n}\backslash
D_{n+1} $, $\partial V$ is decorated by an $\varepsilon _{n}$ dense finite
subset (the endpoints of the crosscuts thus far selected) and hence we may
apply Theorem \ref{markedchop} to $V$ to obtain the following Theorem.

\begin{theorem}
\label{main2}Suppose $X\subset R^{2}$ is a nonseparating continuum. Then
there exists a null sequence of crosscuts $\{\beta _{n}\}$ such that $%
int(\beta _{n})\subset R^{2}\backslash X$ and $\partial \beta _{n}\subset X$
and such that $int(\beta _{n})\cap int(\beta _{m})=\emptyset $ if $m\neq n$
and such that the components $\{U_{n}\}$ of $R^{2}\backslash (X\cup (\cup
\beta _{n}))$ form a null sequence of open sets (with simple closed curve
boundaries), and for each $U_{n},$ the simple closed curve $\partial U_{n}$
is the finite union of concatenated arcs $\beta _{n}$.
\end{theorem}

\section{Main results}

\begin{theorem}
\label{main1}Suppose $X\subset R^{2}$ is compact and $U=R^{2}\backslash X$
is connected. Suppose $X$ has at least two components. Then there exists a
sequence of arcs $\alpha _{1},\alpha _{2},...$ such that $l(\alpha
_{i})\rightarrow 0,$ and $Z=X\cup $ $(\cup _{i=1}^{\infty }\alpha _{i})$ is
cellular, and $int(\alpha _{i})\subset U,$ and $\alpha _{i}$ connects
distinct components of $X,$ and $int(\alpha _{i})\cap int(\alpha
_{j})=\emptyset $ if $i\neq j.$
\end{theorem}

\begin{proof}
Let $\delta _{n}=\frac{1}{10^{n}}.$

Apply Theorem \ref{Sn} to obtain a sequence of closed sets $S_{n}\subset
R^{2}$ such that $S_{n}$ is the union of finitely many pairwise disjoint
closed PL topological disks, such that $S_{n+1}\subset int(S_{n})$, such
that $X=\cap _{n=1}^{\infty }S_{n}$ and such that $N(S_{n},S_{n+1})<\delta
_{n}$ and such that $\lim_{n\rightarrow \infty }M(S_{n},S_{n+1})=0.$ We
require that $S_{1}$ is a connected PL disk.

Name a sequence $\{\varepsilon _{n}\}$ such that $\varepsilon
_{n}>2M(S_{n},S_{n+1})+2\delta _{n}$ and such that $\varepsilon
_{n}\rightarrow 0.$

Let $Y_{1}=\emptyset $ and proceed recursively as follows.

Suppose $Y_{n}=\{y_{n}^{1},...y_{n}^{k_{n}}\}\subset \partial S_{n}$ is
finite. Apply Theorem \ref{shrtarcs} to obtain pairwise disjoint PL closed
arcs $\{\gamma _{n}^{i}\}\subset S_{n}$ such that $int(\gamma
_{n}^{i})\subset int(S_{n})\backslash S_{n+1}$ and $\gamma _{n}^{i}$
connects $y_{n}^{i}$ to $S_{n+1}$ and $l(\gamma _{i})<\delta _{n}.$

Recall each component $P_{i}^{n}\subset S_{n}$ is a PL disk, and the
subspace $S_{n+1}\cup _{i,j}\{\gamma _{j}^{i}\}\subset S_{n}$ is the union
of pairwise disjoint PL cellular sets.

Now apply Theorem \ref{main3} to the data at hand as follows.

Apply the algorithm in section \ref{alg} to the data $(S_{n},S_{n+1}\cup
_{i,j}\{\gamma _{j}^{i}\})$ creating a finite sequence of arcs $\{\alpha
_{n}^{i}\}\subset S_{n}$ with the following properties: $l(\alpha
_{n}^{i})\leq 2M(S_{n},S_{n+1}),$ and $\alpha _{n}^{i}\cap \alpha _{n}^{j}$
is connected, and if $i\neq j$ then $\alpha _{n}^{i}\cap \alpha _{n}^{j}$
does not disconnect $\alpha _{n}^{i},$ and $\alpha _{n}^{i}$ connects
distinct components of $S_{n+1}\cup _{i,j}\{\gamma _{j}^{k}\}.$

Next apply the constructions in section \ref{perturbarc}, replacing the arcs 
$\alpha _{n}^{i}$ with pairwise disjoint arcs $\beta _{n}^{i}\subset S_{n}$ $%
\ $(replacing the notation $\alpha _{n}^{\ast \ast \ast i}$) with all of the
following properties:

$l(\beta _{n}^{i})<2M(S_{n},S_{n+1})+2\delta _{n},$ $\beta _{n}^{i}\cap
\gamma _{k}^{j}=\emptyset ,$ and $\beta _{n}^{i}$ connects the same two
distinct components of $S_{n+1}\cup _{i,j}\{\gamma _{j}^{k}\}$ ( as $\alpha
_{n}^{i}$), and $int(\beta _{n}^{i})\subset int(S_{n}),$ and each component
of $S_{n+1}\cup (\cup _{i}\beta _{n}^{i})$ is cellular.

Now let $Y_{n+1}=\partial S_{n+1}\cap ((\cup _{i}\beta _{n}^{i})\cup (\cup
_{i}\gamma _{n}^{i})$ and repeat the construction.

(It is allowed at a given stage $n$, that $\cup _{i}\beta _{n}^{i}=\emptyset
,$ (in the event that $S_{n}$ and $S_{n+1}$ have the same number of
components, and in fact this behavior is inevitable if $X$ has finitely many
components)).

To understand the components of $Z\backslash X,$ by construction at each
stage $n,$ new arcs $\{\beta _{n}^{i}\}$ are created such that $l(\beta
_{n}^{i})<\varepsilon _{n}.$ In subsequent stages a given end of $\beta
_{n}^{i}$ will be lengthened by attaching concatenated arcs $\gamma
_{n+1}^{(n,i)}\cup \gamma _{n+2}^{(n,i)}...$ (and on the other end of $\beta
_{n}^{i}$ we have concatenated arcs $\gamma _{n+1}^{\ast (n,i)}\cup \gamma
_{n+2}^{\ast (n,i)}...$ ).

Thus, $\beta _{n}^{i}\cup \gamma _{n+1}^{(n,i)}\cup \gamma
_{n+2}^{(n,i)}...\cup (\gamma _{n+1}^{\ast (n,i)}\cup \gamma _{n+2}^{\ast
(n,i)}...)$ is an open arc with Euclidean pathlength \ less than $%
2\varepsilon _{n}+\Sigma _{k=n}^{\infty }\frac{2}{10^{k}}<2\varepsilon _{n}+%
\frac{1}{2^{n}}.$

Define $\kappa _{n}^{i}=\overline{\beta _{n}^{i}\cup \gamma
_{n+1}^{(n,i)}\cup \gamma _{n+2}^{(n,i)}...\cup (\gamma _{n+1}^{\ast
(n,i)}\cup \gamma _{n+2}^{\ast (n,i)}...)}.$ Note the ends of the open arc $%
\beta _{n}^{i}\cup \gamma _{n+1}^{(n,i)}\cup \gamma _{n+2}^{(n,i)}...\cup
(\gamma _{n+1}^{\ast (n,i)}\cup \gamma _{n+2}^{\ast (n,i)}...)$ converge
since this open arc has finite geometric length.

Moreover, since the extended ends of $\beta _{n}^{i}$ will be forever
trapped in distinct components of $S_{n},$ $\kappa _{n}^{i}$ is a closed
arc, (as a opposed to a simple closed curve).

Define $Z=X\cup (\cup _{n.i}\kappa _{n}^{i}).$

Recall $\gamma _{n}^{j}\subset S_{n},$ and $M(S_{n},S_{n+1})\rightarrow 0$
and $X=\cap S_{n}.$ Thus $\partial \kappa _{n}^{i}\subset X.$

By construction, $Z$ is the nested intersection of the cellular sets $%
S_{n+1}\cup ((\cup _{k\leq n}\beta _{k}^{i})\cup (\cup _{k\leq n}\gamma
_{k}^{i})).$ Consequently $Z$ is cellular.

By construction $int(\kappa _{n}^{i})\cap int(\kappa _{m}^{j})=\emptyset $
(if $n\neq m$ or $i\neq j$).

Reindex the arcs doubly indexed sequence $\{\kappa _{n}^{i}\}$ as $\alpha
_{1},\alpha _{2},...$ to obtain the desired arcs.
\end{proof}

If $X\subset S^{2}$ is compact then $S^{2}\backslash X$ has at most
countably many components $U_{1},U_{2},..$, and we apply Theorem \ref{main1}
to each component $U_{n}\subset S^{2}\backslash X$ to obtain the following
result applicable to all planar compacta.

\begin{corollary}
\label{maincor}Suppose $X\subset R^{2}$ is compact. Then there exists a
sequence of closed arcs $\alpha _{1},\alpha _{2},..$ such that $int(\alpha
_{i})\subset R^{2}\backslash X,$ and $\partial \alpha _{i}\subset X$ and $%
\lim_{i\rightarrow \infty }l(\alpha _{i})=\emptyset $, and $int(\alpha
_{i})\cap int(\alpha _{j})=\emptyset $ if $i\neq j,$ and if $Y=X\cup (\cup
\alpha _{i})$) and if $U$ is a component of $R^{2}\cup \{\infty \}\backslash
Y$ then $U$ is simply connected, and each component of $R^{2}\backslash X$
contains precisely one component of $R^{2}\backslash Y.$
\end{corollary}

\begin{theorem}
\label{main4}Suppose $X\subset R^{2}$ is a Peano continuum and $Y\subset
R^{2}$ is any set. Suppose $p\in X$ and $f,g:X\rightarrow Y$ are maps such $%
f(p)=g(p).$ Then $f$ and $g$ are homotopic if and only if $f_{\ast }=g_{\ast
}$ (and $f_{\ast }:\pi _{1}(X,p)\rightarrow \pi _{1}(Y,f(p))$ denotes the
induced homomorphism between fundamental groups.)
\end{theorem}

\begin{proof}
If $f\symbol{126}g$ it is immediate that $f_{\ast }=g_{\ast }.$ Conversely
suppose $f_{\ast }=g_{\ast }.$ Our goal is to prove that $f$ and $g$ are
homotopic, and throughout the proof we will replace $f$ by a homotopic map $%
f^{\symbol{94}}$ with nicer properties, and for convenience we will then
rename $f=f^{\symbol{94}}.$

We reduce to the assumption that $X$ is standard as follows.

Apply Lemma \ref{stand} to obtain a standard Peano continuum $Z$ and a
retraction $r:Z\rightarrow X$ such that $r$ is homotopic to $id_{Z}.$
Moreover $(fr)_{\ast }=(gf)_{\ast }$ and we hope to prove $fr$ is homotopic
to $gr.$ If we find such a homotopy from $Z$ then we can restrict to $X$ to
obtain a homotopy from $f$ to $g.$ Thus, renaming $Z$ as $X,$ we have
reduced to the special case that $X$ is a standard Peano continuum.

Suppose $X\subset R^{2}$ is a standard Peano continuum. Let $U\subset
R^{2}\backslash X$ denote the union of those bounded components $%
\{U_{n}\}\subset R^{2}\backslash X$ such that $f_{\partial U_{n}}$ is
inessential in $Y.$ Let $Z=X\cup U.$ By Lemma \ref{circlekill} $Z$ is a
standard Peano continuum and there exists a map $F:Z\rightarrow Y$ such that 
$F_{X}=f$ ( in particular $F_{C_{n}}$ is essential for all isolated circles $%
C_{n}\subset Fr(Z)$)

In similar fashion we can construct a map $G:Z\rightarrow Y$ such that $%
G_{X}=g.$

By Lemma \ref{fg} $F_{\ast }=G_{\ast }$ and if we can prove $F$ is homotopic
to $G$ then, restricting the homotopy to $X,$ we will have a homotopy from $%
f $ to $g.$

Thus, once again renaming $Z=X,$ we have reduced the problem to the further
specialized assumption that $f(C)$ is essential in $Y$ for all isolated
boundary circles $C\subset Fr(X)$ (and $X\subset R^{2}$ is a standard Peano
continuum).

Let $C_{1},C_{2}...$ denote the isolated circles of $Fr(X)$ and for each $n$
select a basepoint $x_{n}\in C_{n}.$

Let $y_{n}=g(x_{n}).$ By uniform continuity $diam(g(C_{k}))\rightarrow 0$
and $diam(f(C_{k}))\rightarrow 0.$ Since $f_{\ast }=g_{\ast }$ we know the
unbased loops $f_{C_{k}}$ and $g_{C_{k}}$ are essential and freely homotopic
in $Y.$

Thus it follows from Lemma \ref{shape} that $d(x_{n},y_{n})\rightarrow 0.$
Now apply Theorem \ref{pointadjust} to obtain a map $f^{\symbol{94}%
}:X\rightarrow Y$ such that $f^{\symbol{94}}$ is homotopic to $f$ and such
that $f^{\symbol{94}}(x_{n})=g^{\symbol{94}}(x_{n}).$

Thus, wolog we may rename $f=f^{\symbol{94}}$ and we assume henceforth that $%
f(x_{n})=g(x_{n}).$

Now apply Lemma \ref{frfix} to obtain a map $f^{\symbol{94}}:X\rightarrow Y$
such that $f_{Fr(X)}^{\symbol{94}}=g_{Fr(X)}$ and such that $f$ and $f^{%
\symbol{94}}$ are homotopic.

Once again, we may rename $f^{\symbol{94}}=f$ and henceforth assume $%
f_{Fr(X)}=g_{Fr(X)}.$

Apply Corollary \ref{maincor} to obtain a null sequence of arcs $\beta
_{1},\beta _{2},...$ such that $int(\beta _{n})\subset int(X)$ and $\partial
\beta _{n}\subset Fr(X)$ , and $int(\beta _{n})\cap int(\beta
_{m})=\emptyset $ if $n\neq m,$ and so that the components of $%
int(X)\backslash (\beta _{1}\cup \beta _{2}...)$ are simply connected open
planar sets and so that the endpoints of $\beta _{n}$ belong to distinct
components of $Fr(X).$

Let $Z=Fr(X)\cup \beta _{1}\cup \beta _{2}.....$

For each bounded component $U_{n}\subset R^{2}\backslash Z$, apply Theorem 
\ref{main2} to obtain a null sequence of crosscuts $\alpha _{1}^{n},\alpha
_{2}^{n},..$ so that $int(\alpha _{k}^{n})\subset U_{n}$ and $\partial
\alpha _{k}^{n}\subset \partial U_{n}$ and such that $int(\alpha
_{k}^{n})\cap int(\alpha _{m}^{n})=\emptyset $ if $k\neq n$ and such that
the components $\{V_{n}\}$ of $\overline{U_{n}}\backslash (\alpha
_{1}^{n},\alpha _{2}^{n},..)$ form a null sequence of open sets (with simple
closed curve boundaries), and for each $B_{n},$ the simple closed curve $%
\partial V_{n}$ is the finite union of concatenated arcs $\alpha _{n}^{k}$.

Now let $\alpha _{1},\alpha _{2},...$ denote the arcs $\cup _{m}\{\beta
_{m}\}\cup _{n,k}(\alpha _{k}^{n}).$ By construction $\partial \alpha
_{n}\subset Fr(X)$ and $int(\alpha _{n})\cap int(\alpha _{m})=\emptyset $ if 
$m\neq n$ and $\{\alpha _{n}\}$ is a null sequence of arcs. Let $Z=X\cup
\alpha _{1}\cup \alpha _{2}....$

Apply Lemma \ref{arcfix} to obtain a map $f^{\symbol{94}}:X\rightarrow Y$
such that $f^{\symbol{94}}$ is homotopic to $f$ and $f_{Z}^{\symbol{94}%
}=g_{Z}.$ As before rename $f=f^{\symbol{94}}.$

It follows now from Lemma \ref{cherry} that $f$ and $g$ are homotopic.
\end{proof}

Consequently we obtain the following version of Whitehead's Theorem for
planar Peano continua.

\begin{corollary}
\label{White}If $X,Y\subset R^{2}$ are Peano continua, then a map $%
f:X\rightarrow Y$ is a homotopy equivalence if there exists $p\in X$ and $%
q\in Y$ such that $f_{\ast }:\pi _{1}(X,p)\rightarrow \pi _{1}(Y,q)$ is an
isomorphism and such that $f_{\ast }^{-1}:\pi _{1}(Y,q)\rightarrow \pi
_{1}(X,p)$ is induced by a map.
\end{corollary}

\begin{proof}
Let $g:(Y,q)\rightarrow (X,p)$ be such that $(fg)_{\ast }=f_{\ast }g_{\ast
}=id_{\pi _{1}(Y,q)}$ and $(gf)_{\ast }=g_{\ast }f_{\ast }=id_{\pi
_{1}(X,p)}.$ Then by Theorem \ref{main4} $gf$ and $fg$ are homotopic to the
respective identities. Hence $f$ is a homotopy equivalence.
\end{proof}

\end{document}